\RequirePackage{fix-cm}
\documentclass{article}
\usepackage{graphicx}
\usepackage[utf8]{inputenc}
\usepackage[T1]{fontenc}

\usepackage{amsmath} %mehr Mathematikfunktionen
\usepackage{amssymb} %mehr Mathematikfunktionen
\usepackage{mathtools} %enthält coloneqq
\usepackage{mathrsfs}
\usepackage{bm}

\usepackage{hyperref}	%Links erstellen
\usepackage[english]{cleveref} %cleveres Referenzieren, erstellt Worte wie "`Abschnitt"' automatisch
\usepackage{autonum}

\usepackage{pgfplots}
\pgfplotsset{%cycle list name=black white,
	cycle list name=exotic,
	xlabel near ticks,
	ylabel near ticks,
	width= .5\textwidth,
	height=4.5cm,
	tick label style={font=\scriptsize},
	xlabel style={font=\scriptsize},
	legend style={font=\scriptsize},
	}
\pgfplotsset{compat=1.13}
\usepackage{pgfplotstable}
\usepackage{booktabs}

\usepackage{algorithm}
\usepackage{algpseudocode}
\algnewcommand\algorithmicinput{\textbf{Input:}}
\algnewcommand\Input{\item[\algorithmicinput]}
\algnewcommand\algorithmicoutput{\textbf{Output:}}
\algnewcommand\Output{\item[\algorithmicoutput]}

\DeclareMathOperator{\diag}{diag}
\DeclareMathOperator{\spa}{span}
\DeclareMathOperator{\Real}{Re}
\DeclareMathOperator{\Imag}{Im}
\DeclareMathOperator{\im}{im}
\DeclareMathOperator{\chol}{chol}

\newcommand{\p}[1]{\mathfrak{ #1}}
\newcommand{\B}{\bm{b}}
\newcommand{\customqed}{\qed}

\usepackage{amsthm}
\theoremstyle{plain}
\newtheorem{theorem}{Theorem}[section]
\newtheorem{lemma}{Lemma}[section]

\theoremstyle{remark}
\newtheorem{remark}{Remark}[section]

\theoremstyle{definition}
\newtheorem{definition}{Definition}[section]
\renewcommand{\customqed}{}

\usepackage[small]{titlesec}

\newenvironment{acknowledgements}[0]{ \par \mbox{} \par \noindent \small \textbf{Acknowledgements}}{}
\begin{document}

\title{Riccati ADI: Existence, uniqueness and new iterative methods
}
\author{Christian Bertram\textsuperscript{1} and Heike Faßbender\textsuperscript{1}}
\date{\textsuperscript{1}Institute for Numerical Analysis, Technische Universität Braunschweig, Germany}

\maketitle

\begin{abstract}
	The approximate solution of large-scale algebraic Riccati equations is considered. We are interested in approximate solutions which yield a Riccati residual matrix of a particular small rank. It is assumed that such approximate solutions can be written in factored form $ZYZ^*$ with a rectangular matrix $Z$ and a small quadratic matrix $Y$. We propose to choose $Z$ such that its columns span a certain rational Krylov subspace. Conditions under which such an approximate solution exists and is unique are determined. It is shown that the proposed method can be interpreted as an oblique projection method. Two new iterative procedures with efficient updates of the solution and the residual factor are derived. With our approach complex system matrices can be handled, realification is provided and parallelization is introduced.

\end{abstract}

\section{Introduction}
Consider the continuous-time algebraic Riccati equation
\begin{align}
	\begin{split}
	\label{def:are}
	0 = \mathcal{R}(X) &\coloneqq A^{*}X + XA + C^*C - XBB^{*}X\\
	&= (A^*-XBB^*)X + X(A-BB^*X) + C^*C + XBB^*X
	\end{split}
\end{align}
with complex system matrices $A\in \mathbb{C}^{n\times n}$ large and sparse and $B\in \mathbb{C}^{n\times m}$, $C\in \mathbb{C}^{p\times n}$. This Riccati equation occurs in several applications such as linear-quadratic regulator problems of optimal control, in $\mathcal{H}_2$ and $\mathcal{H}_\infty$ control, in balancing-related model order reduction or in Kalman filtering. Due to the quadratic term there may be many solutions of \eqref{def:are}, but in most applications one is interested in the unique stabilizing solution. This is the positive semidefinite solution $X$ for which $A-BB^*X$ is stable, i.e. all its eigenvalues are contained in the left half of the complex plane. Existence of this solution is guaranteed if $(A,B)$ is stabilizable and $(C,A)$ is detectable \cite[Thm.2.21]{bini11}, \cite[Thm. 9.1.2]{lancaster95}. If $(C,A)$ is observable, then the sought solution is even positive definite \cite[Thm. 9.1.5]{lancaster95}.

We are concerned with the large-scale setting where direct methods for the solution of \eqref{def:are} are infeasible. In this setting, one is interested in a Hermitian low-rank approximation of the form $\tilde X = ZYZ^{*}\approx X$ where $Z$ is a rectangular matrix with only few columns and $Y$ is a small square Hermitian matrix. This form of approximation gives two degrees of freedom: The approximation space, that is, the space spanned by the columns of $Z$, and the choice of the approximate solution in this space which is determined by $Y$. There are several methods which produce such a low-rank approximation; see, e.g. \cite{benner20_compare} for an overview. Basically, there are three families of methods, described briefly next. All these methods use certain (rational) Krylov subspaces as approximation spaces.

In projection methods, such as the extended \cite{heyouni09} and rational Krylov subspace method \cite{Simoncini2014,Simoncini2016} (denoted \emph{*KSM}), the factor $Z$ is chosen such that its columns are an orthonormal basis for a certain (rational) Krylov subspace. The matrix $Y$ is chosen as the solution of an orthogonal projection of the Riccati equation \eqref{def:are} onto the selected Krylov subspace.

Low-rank Newton-Kleinman methods treat \eqref{def:are} as a nonlinear equation and apply Newton's method, so in each iteration step a linear matrix equation has to be solved, see e.g. \cite{benner16_newtonADI} and the references therein. Recently, the projected Newton-Kleinman method was derived in \cite{palitta19_arxiv} which is a projection scheme where the inner linear solves are performed only implicitly, making this approach competitive to the other methods mentioned here.

The third family are the ADI-type methods which are investigated in this work. They are well known from linear matrix equations and derived from the alternating directions implicit method \cite{wachspress88,li2002,kuerschner16_phd}. In order to be able to apply these methods to large-scale problems, the variant of the ADI iteration for dense matrices is reformulated to obtain a low-rank formulation \cite{li2002}. However, due to the nonlinearity of the Riccati equations, at first glimpse, this technique seems to be infeasible here. Though there are ADI-type methods for the Riccati equation, e.g. the qADI algorithm \cite{wong06,wong07}, the algorithm of Amodei and Buchot \cite{amodei10}, the RADI iteration \cite{benner18} and the Cayley subspace iteration \cite{simoncini15}. In fact all these methods produce (at least theoretically) the same approximate solution, which was shown laboriously e.g. in \cite{benbu16,benner18}. The term \emph{Riccati ADI methods} (proposed in \cite{benner18}) will be used throughout this work to denote the ADI-type methods for Riccati equations. In contrast to the other two families the approximation space is not chosen directly but turns out to be a rational Krylov subspace which is determined by shift parameters, see \cite[Prop. 2]{benner18}. Besides, the residual matrix is of rank $p$, see e.g. \cite[Prop. 1]{benner18}.

As discussed in \cite[Sec. 5]{benner18}, the RADI iteration is computationally more efficient than any of the other Riccati ADI methods considered. This method has been formulated for the case of real system matrices $A,$ $B$ and $C$ only. It can handle generalized Riccati equations and provides a low-rank formulation $ZZ^{*}$ for the approximate solution of \eqref{def:are}.
The method offers a shift strategy for fast convergence and techniques to reduce the use of complex arithmetic in case of complex shift parameters. In each iteration step the Sherman-Morrison-Woodbury (SMW) formula \cite[Chp. 2.1.4]{golub13} is used to solve the occurring linear system, which becomes expensive in case of a system matrix $B$ with many columns, i.e. with large $m$.

Our ADI-type approach is similar to the approach of the projection methods. We decouple the calculation of a basis $Z$ for a certain (rational) Krylov subspace from the calculation of $Y$. That is, we choose $Z$ and consider the approximate solution $\tilde X=ZYZ^{*}$. The matrix $Y$ is determined such that the \emph{Riccati residual}
\begin{align}
	\mathcal{R}(\tilde X) = A^{*}\tilde X + \tilde XA + C^*C - \tilde XBB^{*}\tilde X
\end{align}
is of rank $p$. Due to its symmetry and its rank it can be factored into $\mathcal{R}(\tilde X)=\tilde R\tilde R^*$ with the \emph{Riccati residual factor} $\tilde R\in \mathbb{C}^{n\times p}$. We call such an approximate solution $\tilde X=ZYZ^{*}$ with a rank $p$ residual an \emph{ADI approximate solution}.

We answer the following questions: For which (rational) Krylov subspaces with basis $Z$ does a small matrix $Y$ exist such that the Riccati residual $\mathcal{R}(ZYZ^*)$ can be factored into rank-$p$ matrices? Is this approximate solution unique? How can it be obtained efficiently?

We show that the ADI approximate solution exists and is unique if the approximation space is a rational Krylov subspace whose poles satisfy a simple condition. Generalized Riccati equations can be handled.
It is demonstrated that our Riccati ADI method can be interpreted as an oblique projection method.
The projected system matrices are connected to the poles and zeros of a rational function corresponding to the residual factor. We derive two efficient iterative methods, which, due to our uniqueness result, generate the same approximate solution to \eqref{def:are} as the other Riccati ADI methods. When all system matrices are real valued and all poles of the rational Krylov subspace appear in complex conjugated pairs, then we can choose a real basis $Z$ and all involved quantities remain real valued. Our approach allows to parallelize linear systems solves which are necessary to generate the Krylov basis $Z$.

To answer the above questions and derive the new iterative methods we make use of rational Arnoldi decompositions which are introduced in \Cref{sec:rad}. In \Cref{sec:approx_riccati} conditions for the existence and uniqueness of the ADI approximate solution are derived and the connection to projection methods is presented. New iterative methods based on these theoretical findings are derived in \Cref{sec:iterexpa}. Further, realification and parallelization are introduced here. In \Cref{sec:discussion} we indicate how to deal with generalized Riccati equations, show how our method simplifies in case of linear matrix equations and mention a shift selection strategy. In the numerical experiments in \Cref{sec:numexp} we compare our new iterative methods and demonstrate the effects of parallelization. Concluding remarks are given in \Cref{sec:conclusion}.

\subsection{Notation}
An underscore is used to indicate matrices with more rows than columns. With a negative index we denote the quadratic lower submatrix $\underline{K_{-j}}=\begin{bmatrix} 0 & I\end{bmatrix}\underline{K_j}$ of such a rectangular matrix. A (block) diagonal matrix with diagonal entries $A_1, \dots, A_l$ is indicated by $\diag(A_1, \dots, A_l)$. The set of eigenvalues of a square matrix $A$ is given by $\Lambda(A)$. The matrix $A^*$ is the complex conjugated of the transposed matrix $A^{\mathsf{T}}$. We use this notation also for real matrices. A linear space $\mathcal{V}$ is called \emph{$A$-variant} if $A\mathcal{V}\not\subseteq \mathcal{V}$ and \emph{$A$-invariant} otherwise. With $\mathbb{C}_{+}$ ($\mathbb{C}_{-}$) we indicate the set of complex numbers with positive (negative) real part. For $j\in \mathbb{N}_0$ a polynomial $\p{p}$ is given via $\p{p}(x)=\sum_{i=0}^{j}a_ix^{i}$ with coefficients $a_i\in \mathbb{C}$ for $i=0, \dots, j$. The degree $\deg(\p{p})$ of the polynomial $\p{p}$ is the largest index belonging to a nonzero coefficient. If all coefficients are zero, then the degree is set to $-\infty$. The set of all polynomials with degree at most $j$ is denoted by $\Pi_j$. The $i$-th canonical standard basis vector is denoted $e_i$.

\section{Rational Krylov subspaces and rational Arnoldi decompositions}\label{sec:rad}
We start with a discussion on (rational) Krylov subspaces. The definitions and theorems in this section are largely based on the work of Berljafa \cite{berljafa15,berljafa17}.
Given a matrix $A\in \mathbb{C}^{n\times n}$, a vector $b\in \mathbb{C}^{n}$ and a number $j\in \mathbb{N}$, the \emph{polynomial Krylov subspace of order $j$} is defined by
\begin{align}\label{eq:polkry}
	\mathcal{K}_j(A,b)&\coloneqq\spa\!\left\{ b, Ab, \dots, A^{j-1}b \right\} \\
	&= \{ \p{p}(A)b \mid \p{p}\in \Pi_{j-1} \},
\end{align}
the set of all polynomials of degree at most $j-1$ in $A$ multiplied with $b$. We extend this definition to rational functions instead of only polynomials. Let $\p{q}\in \Pi_j$ be a fixed polynomial of degree at most $j$ with no roots in $\Lambda(A)$ so $\p{q}(A)$ is invertible. The \emph{rational Krylov subspace of order $j$ associated to $\p{q}$} is defined by
\begin{align}\label{eq:ratkry}
	\mathcal{K}_j(A,b,\p{q})&\coloneqq \p{q}(A)^{-1}\mathcal{K}_j(A,b)\\
	&= \{ \p{q}(A)^{-1}\p{p}(A)b \mid \p{p}\in \Pi_{j-1} \},
\end{align}
the set of all rational functions in $A$ multiplied with $b$ with a nominator polynomial of degree at most $j-1$ and fixed denominator polynomial $\p{q}$ of degree equal or less than the order $j$ of the Krylov subspace $\mathcal{K}_j(A,b)$. The polynomial $\p{q}$ has roots $s_1, \dots, s_{\deg(\p{q})}\in \mathbb{C}\setminus \Lambda(A)$. If $\deg(\p{q})<j$ we say it has $j-\deg(\p{q})$ formal roots at infinity. This allows us to characterize the rational Krylov subspace through a set of $j$ roots $s=\left\{ s_1, \dots, s_{j} \right\}\subset \mathbb{C}\cup \{\infty\} \setminus \Lambda(A)$ instead of the polynomial $\p{q}$, i.e. we can write $\mathcal{K}_j(A,b,s)$ for \eqref{eq:ratkry}. The roots of $\p{q}$ are also called \emph{shifts} or \emph{poles} of the rational Krylov subspace.

The polynomial Krylov subspace \eqref{eq:polkry} of order $j$ is a special case of a rational Krylov subspace with $\p{q}=1$, i.e. a rational Krylov subspace of order $j$ with $j$ poles at infinity. On the other extreme, if $\p{q}$ has degree $j$, then $\p{q}^{-1}\p{p}$ is a proper rational function as the degree of $\p{p}$ is at most $j-1$ and so $b\not\in\mathcal{K}_j(A,b,\p{q})$ holds. In between these two cases we may apply polynomial long division to decompose the rational Krylov subspace. Let $\p{p}\in \Pi_{j-1}$, $\p{q}\in \Pi_{j}$, and set $l=j-\deg(\p{q})$, the number of poles at infinity of a rational Krylov subspace of order $j$ associated to $\p{q}$. We find $\p{u}\in \Pi_{l-1}$ and $\p{r}\in\Pi_{\deg(\p{q})-1}$ with $\p{p}=\p{q}\p{u}+\p{r}$ which results in the decomposition
\begin{align}
	\mathcal{K}_j(A,b,\p{q})&= \{\p{u}(A)b + \p{q}(A)^{-1}\p{r}(A)b \mid \p{u}\in \Pi_{l-1},\, \p{r}\in \Pi_{\deg(\p{q})-1} \}\\
	&= \mathcal{K}_{l}(A,b) + \mathcal{K}_{\deg(\p{q})}(A,b,\p{q}) \label{eq:ratpolydec}
\end{align}
of a rational Krylov subspace into a polynomial Krylov subspace of order $l$ and a rational Krylov subspace of order $\deg(\p{q})$ without poles at infinity.

In particular it holds $\mathcal{K}_l(A,b)\subset \mathcal{K}_{j}(A,b,s)$ if $s$ contains the shift infinity at least $l$ times. More specifically $b$ is contained in a rational Krylov subspace if the shift infinity occurs at least once. We often omit the term \emph{rational} in the following.

Consider the general case of an $A$-variant Krylov subspace $\mathcal{K}_j(A,b,s)$, that is, $A\mathcal{K}_j(A,b,s)\not\subset\mathcal{K}_j(A,b,s)$. Then the spaces $\mathcal{K}_j(A,b,s)$ and $A\mathcal{K}_j(A,b,s)$ are of dimension $j$ and their intersection is of dimension $j-1$, i.e. $\mathcal{K}_j(A,b,s)$ is almost $A$-invariant. This implies that the sum of these two spaces is of dimension $j+1$. More specifically it holds
\begin{align}
	A\mathcal{K}_j(A,b,s) &= A\mathcal{K}_j(A,b,\p{q})\\
	&= \p{q}(A)^{-1}A\mathcal{K}_j(A,b)\\
	&\subset \p{q}(A)^{-1}\mathcal{K}_{j+1}(A,b)\\
	&= \mathcal{K}_{j+1}(A,b,\p{q})\\
	&= \mathcal{K}_{j+1}(A,b,s\cup\{\infty\}),\label{eq:augkry}
\end{align}
with one additional shift infinity in the last line because the order of the polynomial Krylov subspace was increased by one without altering $\p{q}$. We call the latter space with an additional infinite shift \emph{augmented Krylov subspace} and use the notation $\mathcal{K}_j^+(A,b,s)\coloneqq \mathcal{K}_{j+1}(A,b,s\cup\{\infty\})$. It contains rational functions in $A$ multiplied with $b$ whose nominator polynomials are of degree at most $j$ instead of only $j-1$ as in \eqref{eq:ratkry}.

Let $V_{j+1}$ be a basis of the augmented space $\mathcal{K}_j^+(A,b,s)$. As it contains both spaces $\mathcal{K}_j(A,b,s)$ and $A\mathcal{K}_j(A,b,s)$ it is possible to encode the effect of multiplication of $\mathcal{K}_j(A,b,s)$ with $A$ in a decomposition of the form
\begin{align}\label{eq:RAD}
	AV_{j+1}\underline{K_{j}} = V_{j+1}\underline{H_{j}}
\end{align}
with rectangular matrices $\underline{K_{j}},\, \underline{H_{j}}\in \mathbb{C}^{(j+1)\times j}$ and
\begin{align}\label{eq:VK}
	\spa\{V_{j+1}\underline{K_{j}}\}&=\mathcal{K}_j(A,b,s),\\
	\spa\{V_{j+1}\underline{H_{j}}\}&=A\mathcal{K}_j(A,b,s).
\end{align}
Decompositions of the form \eqref{eq:RAD}, their connections to Krylov subspaces and their properties play a key role in our discussion on Riccati equations. We therefore introduce some central definitions and results from \cite{berljafa15}.

\begin{remark}
	We note that our definition \eqref{eq:ratkry} of a rational Krylov subspace differs from the definition in \cite{berljafa15} in the order of the involved polynomial Krylov subspace. The rational Krylov subspaces defined in \cite{berljafa15} are equivalent to our augmented Krylov subspace \eqref{eq:augkry}. The reason for the deviation from the literature is that for our purposes Krylov subspaces with only finite shifts are needed, but the augmented Krylov subspaces always contain (at least) one infinite shift. See also \cite[Prop.~3.3]{berljafa17} and the paragraph preceding it.
\end{remark}

\begin{definition}[cf. {\cite[Def. 4.1]{berljafa15}}]
	Let $\underline{K_j},\underline{H_j}\in \mathbb{C}^{(j+1)\times j}$ be rectangular matrices. We say that the pencil $(\underline{H_j},\underline{K_j})$ is \emph{regular} if the lower $j\times j$ subpencil $(\underline{H_{-j}},\underline{K_{-j}})$ is regular, i.e., $\det(z\underline{K_{-j}}-\underline{H_{-j}})$ is not identically equal to zero.
\end{definition}

\begin{definition}[{\cite[Def. 4.2]{berljafa15}}]\label{def:RKD}
	A relation of the form \eqref{eq:RAD} where $V_{j+1}$ is of full column rank and $(\underline{H_{-j}},\underline{K_{-j}})$ is regular is called a \emph{generalized rational Krylov decomposition}. The generalized eigenvalues of $(\underline{H_{-j}},\underline{K_{-j}})$ are called \emph{poles of the decomposition}. If the poles of \eqref{eq:RAD} are outside the spectrum $\Lambda(A)$, then \eqref{eq:RAD} is called a \emph{rational Krylov decomposition (RKD)}.
\end{definition}

If in a (generalized) RKD the matrices $\underline{K_j}$ and $\underline{H_j}$ are upper Hessenberg matrices then the decomposition is called a \emph{(generalized) rational Arnoldi decomposition (RAD)}. The columns of $V_{j+1}$ are called the \emph{basis of the decomposition} and they span the \emph{augmented space of the decomposition} (cf. \cite[Def. 2.3]{berljafa15}). The first column of $V_{j+1}$ is called \emph{starting vector}. Every (generalized) RKD can be transformed to a (generalized) RAD with the same starting vector and the same poles using a generalized Schur form of $(\underline{H_{-j}},\underline{K_{-j}})$ (see \cite[Thm. 4.3]{berljafa15}). If $A$ and $b$ are real valued and the poles appear in complex conjugated pairs then a generalized real Schur form of $(\underline{H_{-j}},\underline{K_{-j}})$ can be used to obtain a \emph{quasi-RAD}, that is $\underline{K_{-j}}$ is a real upper triangular matrix and $\underline{H_{-j}}$ is a real quasi upper-triangular matrix with $1$-by-$1$ and $2$-by-$2$ blocks on the diagonal (cf. \cite[Def. 2.17]{berljafa17}). 
The following theorem guarantees the existence of a RAD for a Krylov subspace \eqref{eq:ratkry}.
\begin{theorem}[cf. {\cite[Thm. 2.5]{berljafa15}}]\label{thm:RADexists}
	Let $\mathcal{V}_{j+1}$ be a vector space of dimension $j+1$ and $s\in \mathbb{C}\cup \{\infty\}$ be a set with $j$ elements. Then $\mathcal{V}_{j+1}=\mathcal{K}_j^+(A,b,s)$ holds if and only if there exists a RAD $AV_{j+1}\underline{K_j}=V_{j+1}\underline{H_j}$ with $\underline{K_j},\,\underline{H_j}\in \mathbb{C}^{(j+1)\times j},$ starting vector $b=V_{j+1}e_1,$ poles $s$ and $\spa\{V_{j+1}\} = \mathcal{V}_{j+1}$.
\end{theorem}
Let $\mathcal{K}_j(A,b,s)$ be a Krylov subspace with associated RKD
\begin{align}\label{eq:RKD}
	AV_{j+1}\underline{K_j}=V_{j+1}\underline{H_j}.
\end{align}
Such a RKD can be transformed by a regular matrix $U\in \mathbb{C}^{(j+1)\times (j+1)}$ into the (generalized) RKD
\begin{align}\label{eq:RADU}
	A\breve V_{j+1}\underline{\breve K_j}=\breve V_{j+1}\underline{\breve H_j}
\end{align}
with $\breve V_{j+1}=V_{j+1}U,$ $\underline{\breve K_j}=U^{-1}\underline{K_j}$ and $\underline{\breve H_j}=U^{-1}\underline{H_j}$. The RKD \eqref{eq:RADU} is associated to the same Krylov subspace $\mathcal{K}_j(A,b,s)$ as \eqref{eq:RKD}. In general the starting vector $\breve b = \breve V_{j+1}e_1 = V_{j+1}Ue_1$ and the poles $\Lambda(\underline{\breve H_{-j}}, \underline{\breve K_{-j}})$ are altered through this transformation and it may happen that the poles coincide with eigenvalues of $A$. Moreover, as $\breve b\in \mathcal{K}_{j}^+(A,b,s)$, there is a polynomial $\p{q}\in \Pi_j$ with roots $s$ and a polynomial $\breve{\p{q}}\in \Pi_j$ such that $\breve b= \p{q}(A)^{-1}\breve{\p{q}}(A)b$. Thus, if no root of $\breve{\p{q}}$ coincides with an eigenvalue of $A$ then $\breve{\p{q}}(A)^{-1}\p{q}(A)\breve b=b$ holds. This implies
\begin{align}
	\mathcal{K}_j(A,b,\p{q})&= \p{q}(A)^{-1}\mathcal{K}_j(A,b)\\
	&= \{ \p{q}(A)^{-1}\p{p}(A)b \mid \p{p}\in \Pi_{j-1} \}\\
	&= \{ \p{q}(A)^{-1}\p{p}(A)\breve{\p{q}}(A)^{-1}\p{q}(A)\breve b \mid \p{p}\in \Pi_{j-1} \}\\
	&= \{ \breve{\p{q}}(A)^{-1}\p{p}(A)\breve b \mid \p{p}\in \Pi_{j-1} \}\\
	&= \mathcal{K}_j(A,\breve b,\breve{\p{q}}),
\end{align}
where we have used the commutativity of rational functions in $A$. Even if a root of $\breve{\p{q}}$ coincides with an eigenvalue of $A$ the reverse result holds.
\begin{theorem}[{cf. \cite[Thm. 4.4]{berljafa15}}]\label{thm:rootfinding}
	Let $\mathcal{V}_{j+1}=\mathcal{K}_{j}^+(A,b,\p{q})$ be $A$-variant. Let $\breve{\p{q}}\in \Pi_j$ be a polynomial with roots equal to the poles of the (generalized) RKD $A\breve V_{j+1}\underline{\breve K_j}=\breve V_{j+1}\underline{\breve H_j}$. If $\breve V_{j+1}$ spans $\mathcal{V}_{j+1}$, then for the starting vector $\breve b=\breve V_{j+1}e_1$ it holds $\breve b=\gamma\p{q}(A)^{-1}\breve{\p{q}}(A)b$ with a scalar $0\neq \gamma\in \mathbb{C}$.
\end{theorem}
Let $0\neq\breve b\in \mathcal{K}_j^+(A,b,\p{q})$ be an arbitrary element of the augmented Krylov subspace with associated RKD \eqref{eq:RKD}. Thus, there is a polynomial $\breve{\p{q}}\in \Pi_j$ with $\breve b=\p{q}(A)^{-1}\breve{\p{q}}(A)b$ and a nontrivial vector $u_1$ with $\breve b=V_{j+1}u_1$. Let $U$ be a regular matrix with $Ue_1=u_1$, so $\breve b=V_{j+1}Ue_1$ holds. Herewith \Cref{thm:rootfinding} allows us to determine the roots of the nominator polynomial $\breve{\p{q}}$ as the poles of the RKD \eqref{eq:RADU}, i.e. the generalized eigenvalues of $(\underline{\breve H_{-j}},\underline{\breve K_{-j}})$.

The definition of rational Krylov subspaces and rational Arnoldi decompositions is generalized in \cite{elsworth20} to the case where the vector $b\in \mathbb{C}^n$ is replaced by a matrix (or \emph{block vector}) $\B\in \mathbb{C}^{n\times p}$. We recall some of the important definitions. The block Krylov subspace or order $j$ is given by
\begin{align}\label{eq:bkrylov}
	\mathcal{K}_{j}^\square(A,\B) &= \text{blockspan}\! \left\{ \B,A\B, \dots, A^{j-1}\B \right\}\\
	&= \left\{ \sum_{k=0}^{j-1} A^k\B C_k \mid C_k\in \mathbb{C}^{p\times p} \right\}.
\end{align}
We only consider the case where the $jp$ columns of $\begin{bmatrix}\B & A\B & \cdots &A^{j-1}\B \end{bmatrix}$ are linearly independent so the block Krylov subspace \eqref{eq:bkrylov} has dimension $jp^2$. Every block vector $\sum_{k=0}^{j-1} A^k\B C_k\in \mathcal{K}_{j}^\square(A,\B)$ corresponds to exactly one matrix polynomial $\sum_{k=0}^{j-1} z^kC_k$.

For the definition of a \emph{block rational Krylov subspace} we again use a polynomial $\p{q}\in \Pi_j$ of degree at most $j$ with no roots in $\Lambda(A)$ and set
\begin{align}
	\mathcal{K}_{j}^\square(A,\B,\p{q}) = \p{q}(A)^{-1}\mathcal{K}_{j}^\square(A,\B).
\end{align}
As in the non-block part we set $\mathcal{K}_{j}^\square(A,\B,s)=\mathcal{K}_{j}^\square(A,\B,\p{q})$ with the set $s\subset \mathbb{C}\cup\left\{ \infty \right\}$ of the $j$ roots of $\p{q}$.

Before we can define block rational Arnoldi decompositions we need the following definition of a block upper-Hessenberg matrix.
\begin{definition}[{cf. \cite[Def. 2.1]{elsworth20}}]
	The block matrix
	\begin{align}
		\underline{H_j} = \begin{bmatrix}H_{11} & \cdots & H_{1p}\\
			H_{21} & \cdots & H_{2p}\\
			& \ddots & \vdots\\
			&&H_{j+1,j}
		\end{bmatrix}\in \mathbb{C}^{(j+1)p\times jp},\ H_{ik}\in\mathbb{C}^{p\times p}
	\end{align}
	is called a \emph{block upper-Hessenberg matrix}. For block upper-Hessenberg matrices $\underline{H_{j}}$ and $\underline{K_{j}}$ the pencil $(\underline{H_{j}},\underline{K_{j}})$ is called an \emph{unreduced block upper-Hessenberg pencil} if one of the subdiagonal blocks $H_{i+1,i}$ or $K_{i+1,i}$ is nonsingular for every $i=1, \dots, j$.
\end{definition}
The next definition generalizes RADs to the block case.
\begin{definition}[{see \cite[Def. 2.2]{elsworth20}}]\label{def:BRAD}
	Let $A\in \mathbb{C}^{n\times n}$. A relation of the form
	\begin{align}
		AV_{j+1}\underline{K_{j}} = V_{j+1}\underline{H_{j}} 
	\end{align}
	is called a \emph{block rational Arnoldi decomposition} (BRAD) if the following conditions are satisfied:
	\begin{enumerate}
		\item $V_{j+1}$ is of full column rank,
		\item $(\underline{H_{j}},\underline{K_{j}})$ is an unreduced block upper-Hessenberg pencil,
		\item $\alpha_i K_{i+1,i}=\beta_iH_{i+1,i}$ holds for some scalars $\alpha_i, \beta_i\in \mathbb{C}$ not both zero, \label{it:BRADab}
		\item the numbers $\mu_i=\alpha_i/\beta_i$ are outside the spectrum $\Lambda(A)$
	\end{enumerate}
	for $i=1, \cdots, j$.
	The numbers $\mu_i\in \mathbb{C}\cup \left\{ \infty \right\}$ are called the \emph{poles} of the BRAD.
\end{definition}
All results presented in the following for Krylov subspaces can be generalized to block Krylov subspaces if not stated otherwise.

\section{Existence and uniqueness of the Riccati ADI solution}\label{sec:approx_riccati}
In this section we derive an existence and uniqueness result for the ADI-type solution of the Riccati equation \eqref{def:are}. For ease of presentation, we first consider the case $C\in \mathbb{C}^{1\times n}$, i.e. $p=1$, and comment on a general $p\in \mathbb{N}$ in \Cref{sec:blockC}.

Consider an approximate solution of \eqref{def:are} of the form $X_j=Z_jY_jZ_j^*$, then the Riccati residual reads
\begin{align}\label{eq:RXj}
	\mathcal{R}(X_j) = A^{*}Z_jY_jZ_j^* + Z_jY_jZ_j^*A + C^*C - Z_jY_jZ_j^*BB^{*}Z_jY_jZ_j^*.
\end{align}
We intend to use a RAD to rewrite \eqref{eq:RXj} in terms of the basis of an augmented Krylov subspace. This will enable us to derive an equation which makes it possible to obtain $Y_j$ such that $\mathcal{R}(X_j)$ is of rank one. Looking at the Riccati residual from the left we see the terms $A^*Z_j,$ $Z_j$ and $C^*$. Due to symmetry, looking from the right we see the adjoint of these terms. Therefore let $Z_j$ be a basis of the Krylov subspace $\mathcal{K}_j(A^*,C^*,s)$ with the matrices $A^*$ and $C^*$ instead of $A$ and $b$ or $\B$. Let
\begin{align}\label{eq:rad_riccati}
	A^*V_{j+1}\underline{K_j} = V_{j+1}\underline{H_j}
\end{align}
be an associated RAD with $Z_j=V_{j+1}\underline{K_j}$, so $X_j = V_{j+1}\underline{K_j} Y_j \underline{K_j}^*V_{j+1}^*$ holds. As $C^*$ is an element of the augmented Krylov subspace there exists a vector $v\in \mathbb{C}^{j+1}$ with 
\begin{align}
	V_{j+1}v=C^{*}.
	\label{eq:v}
\end{align}
Now, with the help of \eqref{eq:rad_riccati} and \eqref{eq:v}, we rewrite the Riccati residual for $X_j$ as follows
\begin{align}
	\mathcal{R}(X_j)&=A^{*}X_j + X_jA + C^*C - X_jBB^{*}X_j\\
	&=A^{*}V_{j+1}\underline{K_j} Y_j \underline{K_j}^*V_{j+1}^* + V_{j+1}\underline{K_j} Y_j \underline{K_j}^*V_{j+1}^*A + V_{j+1}vv^*V_{j+1}^* \\
	&\phantom{=}- V_{j+1}\underline{K_j} Y_j \underbrace{\underline{K_j}^*V_{j+1}^*BB^{*}V_{j+1}\underline{K_j}}_{\eqqcolon S_j} Y_j \underline{K_j}^*V_{j+1}^*\\
	&=V_{j+1}\left(\underline{H_j} Y_j \underline{K_j}^* + \underline{K_j} Y_j \underline{H_j}^* + vv^*- \underline{K_j} Y_j S_j Y_j \underline{K_j}^*\right)V_{j+1}^*\label{eq:VresV}.
\end{align}
The term in brackets is a quadratic $(j+1)\times (j+1)$ matrix. To obtain a $Y_j$ such that the Riccati residual is of rank one we intend to utilize \eqref{eq:VresV} with a special form of the RAD \eqref{eq:rad_riccati}.
\begin{lemma}
	\label{lem:specialRAD}
	Let $\mathcal{K}_j(A^*,C^*,s)$ be an $A^*$-variant Krylov subspace with $\infty\not\in s$. There exists a RAD
\begin{align}\label{eq:RAD_KI}
	A^* V_{j+1}\underline{K_j} = V_{j+1}\underline{{H}_{j}}
\end{align}
	associated to this Krylov subspace with $\underline{K_j}=\begin{bmatrix}0\\ I\end{bmatrix},$ $\underline{H_j}=\begin{bmatrix}h_j\\ \underline{H_{-j}}\end{bmatrix}$, an upper triangular matrix $\underline{H_{-j}}$ and $V_{j+1}=\begin{bmatrix}C^* & Z_j\end{bmatrix}$ with a basis $Z_j$ of $\mathcal{K}_j(A^*,C^*,s)$.
\end{lemma}
\begin{proof}
	As the Krylov subspace $\mathcal{K}_j(A^*,C^*,s)$ is $A^*$-variant the augmented space $\mathcal{K}_j^{+}(A^*,C^*,s)$ is of dimension $j+1$. Thus \Cref{thm:RADexists} guarantees the existence of a RAD
\begin{align}\label{eq:radtilde}
	A^* \tilde V_{j+1}\underline{\tilde K_j} = \tilde V_{j+1}\underline{{\tilde H}_{j}}%
\end{align}
associated to $\mathcal{K}_j(A^*,C^*,s)$ with starting vector $C^*\in \spa\{\tilde V_{j+1}e_1\}$. We aim at constructing an upper triangular matrix $R$ to transform the RAD \eqref{eq:radtilde} into the RAD \eqref{eq:RAD_KI} with $V_{j+1}=\tilde V_{j+1}R,$ $\underline{K_j}=R^{-1}\underline{\tilde K_j}$ and $\underline{H_j}=R^{-1}\underline{\tilde H_j}$.
This regular $R$ is constructed as follows. Let $\alpha\in \mathbb{C}$ so that $C^* = \tilde V_{j+1}\alpha e_1$ and set $R=\begin{bmatrix}\alpha e_1& \underline{\tilde K_j}\end{bmatrix}$ to obtain $\underline{K_j}=\begin{bmatrix}0\\ I\end{bmatrix}$ and $V_{j+1}e_1=C^*$. As we assumed that $\infty\not\in s$ we have, due to \eqref{eq:ratpolydec} and \eqref{eq:VK}, that $C^*\not\in\mathcal{K}_j(A^*,C^*,s)=\spa\{\tilde V_{j+1}\underline{\tilde K_j}\}$. Therefore $e_1\not\in\spa\{\underline{K_j}\}$ and $R$ is indeed regular. As $\underline{\tilde K_j}$ is an upper Hessenberg matrix, $R$ is an upper triangular matrix and so is $\underline{H_{-j}}$. Finally, the structure of $K_{j}$ and \eqref{eq:VK} imply that $Z_j=V_{j+1}\underline{K_j}$ is a basis of $\mathcal{K}_j(A^*,C^*,s),$ which completes the proof. \customqed
\end{proof}

With the assumptions of \Cref{lem:specialRAD} let the RAD \eqref{eq:RAD_KI} be given. In particular hereby $C^*\not\in\mathcal{K}_j(A^*,C^*,s)$ is implied. Due to $V_{j+1}=\begin{bmatrix}C^* & Z_j\end{bmatrix}$ and $\underline{K_j}=\begin{bmatrix}0\\ I\end{bmatrix}$ we have $v=e_1$ in \eqref{eq:v} and $S_j=Z_j^*BB^*Z_j$ in \eqref{eq:VresV}. As $V_{j+1}$ has full rank, we find that the rank of the residual $\mathcal{R}(X_j)$ is the same as the rank of the inner matrix
\begin{align}
	M_j&\coloneqq\underline{H_j} Y_j \underline{K_j}^* + \underline{K_j} Y_j \underline{H_j}^* + e_1e_1^*- \underline{K_j} Y_j S_j Y_j \underline{K_j}^*\\
	&=\begin{bmatrix}0& h_j Y_j\\ 0&\underline{H_{-j}} Y_j\end{bmatrix} + \begin{bmatrix}0&0 \\Y_jh_j^*& Y_j\underline{H_{-j}}^* \end{bmatrix} + \begin{bmatrix} 1 & 0\\ 0& 0 \end{bmatrix} - \begin{bmatrix} 0&0\\ 0& Y_j S_j Y_j \end{bmatrix}\\
	&=\begin{bmatrix} 1 & h_jY_j\\ Y_jh_j^*& \underline{H_{-j}} Y_j+ Y_j\underline{H_{-j}}^* - Y_j S_j Y_j \end{bmatrix}\label{eq:rk1}
\end{align}
from \eqref{eq:VresV}. Thus we are interested in finding $Y_j$ such that $M_j$ is of rank 1. This is the case if and only if
\begin{align}
	M_j &= \begin{bmatrix} 1\\ Y_jh_j^*\end{bmatrix}\begin{bmatrix} 1& h_jY_j\end{bmatrix}\label{eq:mfac}\\
	&=\begin{bmatrix} 1 & h_jY_j\\ Y_jh_j^* & Y_jh_j^*h_jY_j \end{bmatrix}\label{eq:rk2}
\end{align}
holds, which is a rank-1 factorization in terms of the first row and column of $M_j$. From the lower right blocks of \eqref{eq:rk1} and \eqref{eq:rk2} we find that therefore $Y_j$ must be chosen to be a solution of
\begin{align}\label{eq:ssricc}
	\begin{split}
	0 &= \underline{H_{-j}} Y_j+ Y_j\underline{H_{-j}}^* - Y_j S_j Y_j - Y_jh_j^*h_jY_j\\
	&= \underline{H_{-j}} Y_j+ Y_j\underline{H_{-j}}^* - Y_j (S_j + h_j^*h_j)Y_j.
\end{split}
\end{align}
Assume that there exists a full-rank solution $Y_j$ of the homogeneous Riccati equation \eqref{eq:ssricc}, which is desirable as otherwise not all information available in the basis $Z_j$ is incorporated into the ADI approximate solution $X_j$. Such a solution exists if and only if the Lyapunov equation
\begin{align}\label{eq:sslyap}
	0 &= \tilde Y_j\underline{H_{-j}}+ \underline{H_{-j}}^*\tilde Y_j - (S_j + h_j^*h_j)
\end{align}
is solvable with a full-rank solution, and then $Y_j=\tilde Y_j^{-1}$ holds. Thus if the solution of \eqref{eq:sslyap} is of full rank and unique, so is the full-rank solution of \eqref{eq:ssricc}.

The Lyapunov equation \eqref{eq:sslyap} is uniquely solvable with a Hermitian solution if and only if $\underline{H_{-j}}$ and $-\underline{H_{-j}}^*$ have no eigenvalues in common \cite[Thm. 5.2.2]{lancaster95}. Due to \Cref{thm:RADexists} the eigenvalues of $\underline{H_{-j}}$ are equal to the shifts $s$. Therefore the eigenvalue condition $\Lambda(\underline{H_{-j}})\cap \Lambda(-\underline{H_{-j}}^*)=\emptyset$ is equivalent to the shift condition 
$s\cap -\overline s=\emptyset$.

We next proof regularity of the solution $\tilde Y_j$ of \eqref{eq:sslyap}. Assume that $(h_j,\underline{H_{-j}})$ is unobservable, i.e. the observability matrix $\begin{bmatrix}h_j^* & \underline{H_{-j}}^*h_j^* & \cdots & (\underline{H_{-j}}^*)^{j-1}h_j^*\end{bmatrix}^*$ is of rank smaller than $j$. Then due to \cite[Thm. 4.26]{Ant2005} there exists an vector $u\neq 0$ with $h_ju=0$ and $\underline{H_{-j}}u=\mu u$. It holds $\mu\in s$ as the eigenvalues of $\underline{H_{-j}}$ are the poles of the Krylov subspace. Thus we find due to \eqref{eq:RAD_KI}
\begin{align}
	A^{*}V_{j+1}\underline{H_j}u &= A^{*}V_{j+1}\begin{bmatrix}h_j\\ \underline{H_{-j}}\end{bmatrix}u = A^{*}V_{j+1}\begin{bmatrix}0\\ I\end{bmatrix}\mu u\\
	&= \mu V_{j+1}\underline{H_j}u
\end{align}
i.e. $V_{j+1}\underline{H_j}u$ is an eigenvector of $A^*$ with eigenvalue $\mu$. This is a contradiction to the definition of rational Krylov subspaces because the poles of the Krylov subspace must be distinct from the eigenvalues of $A^*$.
Thus $(h_j,\underline{H_{-j}})$ is observable and so $(\underline{H_{-j}}^*,h_j^*)$ is controllable \cite[Thm. 4.23]{Ant2005}. This implies controllability of $\big(\underline{H_{-j}}^*,\begin{bmatrix}h_j^*& Z_j^*B\end{bmatrix}\big)$ where $\begin{bmatrix}h_j^*& Z_j^*B\end{bmatrix}\begin{bmatrix}h_j\\ B^*Z_j\end{bmatrix}=h_j^*h_j + S_j$ holds with $S_j$ from \eqref{eq:VresV}. Herewith regularity of $\tilde Y_j$ can be shown in full analogy to the proof of \cite[Thm. 5.3.1 (b)]{lancaster95}.

Due to \eqref{eq:VresV} and \eqref{eq:mfac} the residual factor $R_j$ of the rank-1 residual $\mathcal{R}(X_j)=R_jR_j^{*}$ is given by
\begin{align}
	R_j&=V_{j+1}\begin{bmatrix}1\\ Y_j h_j^*\end{bmatrix}.
	\label{eq:Rj}
\end{align}
We summarize our findings in the next theorem.
\begin{theorem}
	\label{thm:exiuni}
	Let $\mathcal{K}_j(A^*,C^*,s)$ be an $A^*$-variant Krylov subspace with a basis $Z_j$. Among all matrices $Y_j\in \mathbb{C}^{j\times j}$ such that for $X_j=Z_jY_jZ_j^*$ the residual $\mathcal{R}(X_j)$ is of rank one, there exists a unique regular $Y_j$ if and only if $\infty\not\in s$ and $s\cap -\overline s=\emptyset$.
	Such a unique regular $Y_j$ is determined by the small scale Lyapunov equation
	\begin{align}\label{eq:yeqexiuni}
		0 &= Y_j^{-1}\underline{H_{-j}}+ \underline{H_{-j}}^*Y_j^{-1} - (S_j + h_j^*h_j)
	\end{align}
	where $\underline{H_{-j}}$ and $h_j$ are as in the associated RAD \eqref{eq:RAD_KI} and $S_j$ is as in \eqref{eq:VresV}.
\end{theorem}
\begin{proof}
	We already constructed the unique rank-1 residual approximation for the case $\infty\not\in s$ and $s\cap -\overline s=\emptyset$ preceding this theorem. Now assume $C^*\in\mathcal{K}_j(A^*,C^*,s)$ which means that $s$ does contain the pole infinity.

	As in the proof of \Cref{lem:specialRAD} we can transform a RAD $A^* \tilde V_{j+1}\underline{\tilde K_j} = \tilde V_{j+1}\underline{{\tilde H}_{j}}$ associated to $\mathcal{K}_j(A^*,C^*,s)$ with a regular matrix $R=\begin{bmatrix}w& \underline{\tilde K_j}\end{bmatrix}$ to obtain a RAD with $\underline{K_j}$ and $\underline{H_j}$ as in \eqref{eq:RAD_KI}, but with a different $V_{j+1}$.
Due to $C^*\in \mathcal{K}_j(A^*,C^*,s)=\spa\{V_{j+1}\underline{K_j}\}$ there is a $\tilde v\in \mathbb{C}^{j}$ with $C^*=V_{j+1}\underline{K_j}\tilde v$. Thus we have $v=\begin{bmatrix}0\\ \tilde v\end{bmatrix}$ in \eqref{eq:v} because of the special structure of $\underline{K_j}$.
Again from \eqref{eq:VresV} we find that the rank of the residual is the same as the rank of the inner matrix
\begin{align}
	M_j&=\underline{H_j} Y_j \underline{K_j}^* + \underline{K_j} Y_j \underline{H_j}^* + vv^*- \underline{K_j} Y_j S_j Y_j \underline{K_j}^*\\
	&=\begin{bmatrix}0& h_j Y_j\\ 0&\underline{H_{-j}} Y_j\end{bmatrix} + \begin{bmatrix}0&0 \\Y_jh_j^*& Y_j\underline{H_{-j}}^* \end{bmatrix} + \begin{bmatrix} 0 & 0\\ 0& \tilde v\tilde v^* \end{bmatrix} - \begin{bmatrix} 0&0\\ 0& Y_j S_j Y_j \end{bmatrix}\\
	&=\begin{bmatrix} 0 & h_jY_j\\ Y_jh_j^*&\tilde v\tilde v^* + \underline{H_{-j}} Y_j+ Y_j\underline{H_{-j}}^* - Y_j S_j Y_j \end{bmatrix}.
\end{align}
Assume this matrix is of rank 1. Then due to the first diagonal entry being zero the first row and column have to be zero, i.e. $Y_jh_j^*=0$. This is only possible if $Y_j$ is singular or $h_j=0$. In the latter case the Krylov subspace is $A^*$-invariant. In both cases the assumptions of the theorem are violated and no approximation with the wanted properties exists. \customqed
\end{proof}

\begin{remark}
	Under the conditions of \Cref{thm:exiuni} with $\infty\not\in s$ and $s\cap -\overline{s}=\emptyset$ there exist many approximate solutions $\tilde X_j=Z_j\tilde Y Z_j^*$ of \eqref{eq:VresV} resulting in a rank-1 residual if the full-rank condition on $\tilde Y$ is omitted. Consider e.g. the trivial solution $\tilde Y=0$ of \eqref{eq:ssricc} resulting in the residual factor $C^*$. For other approximations with rank-1 residual consider a subset $\tilde s\subset s$ with $\tilde j$ elements. Clearly $\tilde s$ fulfills the assumptions of \Cref{thm:exiuni} and so there exists a rank-$\tilde j$ approximate solution $\tilde X_j$ resulting in a rank-1 residual. Due to $\mathcal{K}_{\tilde j}(A^*,C^*,\tilde s)\subset \mathcal{K}_j(A^*,C^*,s)$ this solution can be represented by $\tilde X_j=Z_j\tilde Y Z_j^{*}$ with a rank-$\tilde j$ matrix $\tilde Y$. If the shifts in $s$ are pairwise distinct there exist $2^j$ subsets of $s$ and so there are $2^j$ approximate solutions of the form $Z_j\tilde Y Z_j^*$ yielding a rank-1 residual.
\end{remark}

\subsection{Obtaining the Riccati ADI solution via projection}
We now discuss how the Riccati ADI approximate solution can be interpreted as the solution of a projection of the large scale Riccati equation \eqref{def:are} onto a Krylov subspace. For a general projection method let $Z,W\in \mathbb{C}^{n\times j}$ be matrices of rank $j$, i.e. bases of certain $j$ dimensional subspaces, with regular $Z^*W$. Then $\Pi=Z(W^*Z)^{-1}W^*\in \mathbb{C}^{n\times n}$ is a projection onto $\im(\Pi)=\im(Z)$ along $\ker(\Pi)=\ker(W^*)$. Set $\tilde W\coloneqq W(Z^*W)^{-1}$, then $\Pi=Z\tilde W^*$ and it holds $\tilde W^* Z=I$. Let the approximate solution $X_j\approx X$ to the Riccati equation \eqref{def:are} lie in $\im(Z)$ with the representation $X_j=ZY_jZ^*$. Projection of \eqref{def:are} yields the equation
	\begin{align}
		\Pi \mathcal{R}(X_j) \Pi^* = 0
		\label{eq:areP}
	\end{align}
	which is then solved for $Y_j$. Let for instance $Z=W$ be a basis of the Krylov subspace $\mathcal{K}_j(A^*,C^*,s)$, so $\Pi=\Pi^*$ is an orthogonal projection onto the Krylov subspace. This is just the approach used in *KSM.
	
	In the following theorem we show how with the assumptions of \Cref{thm:exiuni} the Riccati ADI approximation can be obtained as the solution of the projected Riccati equation \eqref{eq:areP} using an oblique projection. It is a generalization of \cite[Sec. 3.2]{wolf16}, \cite[Rem. 5.16]{wolf14}, where a similar statement for the ADI iteration to solve Lyapunov equations is presented.
	\begin{theorem}\label{thm:proj}
		Let $\mathcal{K}_j(A^*,C^*,s)$ be an $A^*$-variant Krylov subspace with shifts $s\subset \mathbb{C}$ and $s\cap -\overline s=\emptyset$. Let $A^*V_{j+1}\begin{bmatrix}0\\ I\end{bmatrix}=V_{j+1}\begin{bmatrix}h_j\\ \underline{H_{-j}}\end{bmatrix}$ be an associated RAD as in \Cref{lem:specialRAD} with $V_{j+1}=\begin{bmatrix}C^*& Z_j\end{bmatrix}$. Let $\Pi$ be a projection onto $\im(Z_j)$ along $\ker(W^*)$. If the Riccati residual factor $R_j$ as in \eqref{eq:Rj} is contained in the kernel of the projection $\Pi$, i.e. $\tilde W\perp R_j$, then the projected equation \eqref{eq:areP} is equivalent to the small scale Riccati equation \eqref{eq:ssricc}.
\end{theorem}
\begin{proof}
	From the orthogonality condition of the residual factor $R_j=V_{j+1}\begin{bmatrix}1\\ Y_j h_j^*\end{bmatrix}$ we find
	\begin{align}
		0&=\tilde W^* R_j = \tilde W^{*} \begin{bmatrix}C^* & Z_j\end{bmatrix}\begin{bmatrix}1\\ Y_j h_j^*\end{bmatrix}	 = \tilde W^* (C^* + Z_jY_j h_j^*),
	\end{align}
	or equivalently
	\begin{align}
		 \tilde W^* C^* &=  - \tilde W^*Z_jY_j h_j^* =  - Y_j h_j^*.
	\end{align}
	Herewith we obtain
	\begin{align}
		\tilde W^*V_{j+1}&=\tilde W^*\begin{bmatrix}C^* & Z_j\end{bmatrix} = \begin{bmatrix}\tilde W^*C^* & \tilde W^*Z_j\end{bmatrix} = \begin{bmatrix}-Y_j h_j^* & I\end{bmatrix}.
	\end{align}
	As $Z_j$ is a basis it holds $\ker(Z_j)=\{ 0 \}$, so with $\Pi=Z_j\tilde W^{*}$ the projected equation \eqref{eq:areP} is equivalent to $\tilde W^*\mathcal{R}(X_j)\tilde W=0$. With \eqref{eq:VresV} and \eqref{eq:rk1} we obtain
	\begin{align}
		\tilde W ^*\mathcal{R}(X_j)\tilde W &= \tilde W^* V_{j+1}\begin{bmatrix} 1 & h_jY_j\\ Y_jh_j^*& \underline{H_{-j}} Y_j+ Y_j\underline{H_{-j}}^* - Y_j S_j Y_j \end{bmatrix}V_{j+1}^* \tilde W \\
		&= \begin{bmatrix}-Y_j h_j^* & I\end{bmatrix}\begin{bmatrix} 1 & h_jY_j\\ Y_jh_j^*& \underline{H_{-j}} Y_j+ Y_j\underline{H_{-j}}^* - Y_j S_j Y_j \end{bmatrix}\begin{bmatrix}-h_jY_j \\ I\end{bmatrix}\\
		&= Y_jh_j^*h_jY_j - Y_jh_j^*h_jY_j - Y_jh_j^*h_jY_j + \underline{H_{-j}} Y_j+ Y_j\underline{H_{-j}}^* - Y_j S_j Y_j \\
		&= - Y_jh_j^*h_jY_j + \underline{H_{-j}} Y_j+ Y_j\underline{H_{-j}}^* - Y_j S_j Y_j.
	\end{align}
	This is the right hand side of \eqref{eq:ssricc} which concludes the proof. \customqed
\end{proof}
If in the above theorem $\Pi$ is an orthogonal projection then $\im(\Pi) \perp \ker(\Pi)$ holds. It follows that the conditions $R_j\in \ker(\Pi)$ and $R_j\perp \mathcal{K}_j(A^*,C^*,s)$ are equivalent. Further, the approximations generated by *KSM and the Riccati ADI methods coincide.

Although $\tilde W$ is unknown in practice, the projected system matrices $\tilde W^*A^*Z_j$ and $\tilde W^*(A^*- X_jBB^*)Z_j$ can be expressed in terms of parts of the RAD. This relation is established in the next lemma, which will also be useful in the proofs of the subsequent theorems.
\begin{lemma}\label{lem:proj_sysmat}
Let $\mathcal{K}_j(A^*,C^*,s)$ be an $A^*$-variant Krylov subspace. Let 
\begin{align}\label{eq:RADrat}
	A^*\begin{bmatrix}C^*& Z_j\end{bmatrix}\begin{bmatrix}0\\I\end{bmatrix}=\begin{bmatrix}C^*& Z_j\end{bmatrix}\begin{bmatrix}h_j\\ \underline{H_{-j}}\end{bmatrix}
\end{align}
	be an associated RAD such that \eqref{eq:yeqexiuni} holds, i.e.
	\begin{align}\label{eq:yeqexiuniRADrat}
		0 &= Y_j^{-1}\underline{H_{-j}}+ \underline{H_{-j}}^*Y_j^{-1} - (S_j + h_j^*h_j)
	\end{align}
	as in \Cref{thm:exiuni}. Let $\Pi$ be a projection onto $\im(Z_j)$ along $\ker(W^*)$ with $\tilde W\perp R_j$. Then
\begin{align}
		\tilde W^*A^*Z_j &= -Y_jh_j^*h_j+\underline{H_{-j}}\\
		&= -Y_j\underline{H_{-j}}^*Y_j^{-1} + Y_jS_j
\end{align}
and
\begin{align}
	\tilde W^*(A^*- X_jBB^*)Z_j &= -Y_j\underline{H_{-j}}^*Y_j^{-1}\\
	&= \underline{H_{-j}} - Y_jS_j - Y_jh_j^*h_j
\end{align}
hold.
\end{lemma}
\begin{proof}
With $\tilde W^*C^*=-Y_jh_j^*$ as in the proof of \Cref{thm:proj}, with $\tilde W^* Z_j=I$ and by utilizing the RAD equation \eqref{eq:RADrat} we find
\begin{align}
	-Y_jh_j^*h_j + \underline{H_{-j}} &= \tilde W^{*}C^*h_j + \tilde W^* Z_j\underline{H_{-j}}\\
	&=\tilde W^{*}\begin{bmatrix}C^*& Z_j\end{bmatrix}\begin{bmatrix}h_j\\ \underline{H_{-j}}\end{bmatrix}\\
	&= W^{*}A^* \begin{bmatrix}C^*& Z_j\end{bmatrix}\begin{bmatrix}0\\I\end{bmatrix}\\
	&= \tilde W^{*}A^* Z_j.
\end{align}
Multiplication of \eqref{eq:yeqexiuniRADrat} with $Y_j$ from the left implies 
\begin{align}
-Y_jh_j^*h_j+\underline{H_{-j}}=-Y_j\underline{H_{-j}}^*Y_j^{-1} + Y_jS_j,
\end{align}
which concludes the first part of the proof.

For the second part consider $S_j=Z_j^*BB^*Z_j$ and again $\tilde W^*Z_j=I$ to obtain
	\begin{align}
		Y_jS_j &= \tilde W^*Z_jY_jS_j\\
		&= \tilde W^*Z_jY_jZ_j^*BB^*Z_j\\
		&= \tilde W^*X_jBB^*Z_j.
	\end{align}
	Multiplication of \eqref{eq:yeqexiuniRADrat} with $Y_j$ from the left yields
	\begin{align}
		0 &= \underline{H_{-j}}+ Y_j\underline{H_{-j}}^*Y_j^{-1} - Y_j(S_j + h_j^*h_j).
	\end{align}
	Thus we have
	\begin{align}
		-Y_j\underline{H_{-j}}^*Y_j^{-1} &= \underline{H_{-j}} - Y_jS_j - Y_jh_j^*h_j\\
		&= \tilde W^*A^*Z_j - \tilde W^*X_jBB^*Z_j\\
		&= \tilde W^*(A^*- X_jBB^*)Z_j
	\end{align}
	with $-Y_jh_j^*h_j + \underline{H_{-j}} = \tilde W^*A^*Z_j$ from the first part. This concludes the second part of the proof. \customqed
\end{proof}

The residual factor $R_j$ from \eqref{eq:Rj} is a linear combination of the columns of $V_{j+1}$ and so it is an element of the augmented Krylov subspace $\mathcal{K}_j^{+}(A^*,C^*,s)$, i.e. a rational function in $A^*$ multiplied with $C^*$. We aim at specifying this rational function, which turns out to be connected to the eigenvalues of the projected matrix $\tilde W^*A^* Z_j$ as stated in the next theorem.
\begin{theorem}\label{thm:Rrat}
	Let the same assumptions as in \Cref{lem:proj_sysmat} hold. Additionally let $\mathcal{K}_j^{+}(A^*,C^*,s)$ be $A^*$-variant. Let $\p{p}_j,\p{q}_j\in \Pi_j$ be normalized polynomials of degree $j$ given by
	\begin{align}
		\p{p}_j(x) &= \prod_{i=1}^j \big(x-\lambda_i^{(j)}\big) \text{ and } \p{q}_j(x) = \prod_{i=1}^j (x-s_i)
	\end{align}
	with the eigenvalues $\{\lambda_1^{(j)}, \dots, \lambda_j^{(j)}\} = \Lambda(\tilde W^*A^*Z_j)$ and the poles $s_i\in s$ of the Krylov subspace.
	Then for the residual factor \eqref{eq:Rj}
	\begin{align}
		R_j &= \p{q}_j(A^*)^{-1} \p{p}_j(A^*)C^* = \prod_{i=1}^j \frac{A^*-\lambda_i^{(j)}I_n}{A^*-s_iI_n} C^*
	\end{align}
holds.
\end{theorem}
\begin{proof}
	By construction of the RAD \eqref{eq:RADrat} $Z_j$ is a basis of the Krylov subspace $\mathcal{K}_j(A^*,C^*,s)=\mathcal{K}_j(A^*,C^*,\p{q}_j)$ and $\p{q}_j$ is normalized. Due to the definition of Krylov subspaces \eqref{eq:ratkry} there exists a polynomial $\breve{\p{p}}\in \Pi_{j-1}$ of degree at most $j-1$ with
	\begin{align}
		Z_jY_j h_j^* = \p{q}_j(A^*)^{-1}\breve{\p{p}}(A^*)C^*.
	\end{align}
	Thus for the residual factor \eqref{eq:Rj}
	\begin{align}
		R_j &= \begin{bmatrix}C^*& Z_j\end{bmatrix}\begin{bmatrix}1\\ Y_j h_j^*\end{bmatrix} = C^* + \p{q}_j(A^*)^{-1}\breve{\p{p}}(A^*)C^*\\
		&= \p{q}_j(A^*)^{-1}\big(\p{q}_j(A^*) + \breve{\p{p}}(A^*)\big)C^*
	\end{align}
	holds. Set $\p{p}_j=\p{q}_j + \breve{\p{p}}$, which is a normalized polynomial because $\p{q}_j$ is normalized and of degree $j$ and $\breve{\p{p}}$ is of degree at most $j-1$. It remains to specify the roots of $\p{p}_j$. We transform the RAD \eqref{eq:RADrat} so that it is a (generalized) RKD with starting vector $R_j$ and use \Cref{thm:rootfinding}. In order to do so consider the matrix
\begin{align}
	U = \begin{bmatrix} 1& 0\\Y_jh_j^*& I\end{bmatrix} \text{ with inverse }
	U^{-1} = \begin{bmatrix} 1& 0\\-Y_jh_j^*& I\end{bmatrix}.
\end{align}
Transformation of the RAD \eqref{eq:RADrat} with $U$ yields
\begin{align}
	\begin{bmatrix}C^*& Z_j\end{bmatrix}U = \begin{bmatrix}R_j& Z_j\end{bmatrix},
\end{align}
$\underline{K_j}$ is left unchanged and multiplication of $U^{-1}$ with $\underline{H_j}$ yields
\begin{align}
	U^{-1}\underline{H_{j}} &= U^{-1}\begin{bmatrix}h_j\\ \underline{H_{-j}}\end{bmatrix}
	= \begin{bmatrix}h_j\\ -Y_jh_j^{*}h_j + \underline{H_{-j}}\end{bmatrix}
	= \begin{bmatrix}h_j\\ \tilde W^*A^*Z_j\end{bmatrix},
\end{align}
where the last equality is due to \Cref{lem:proj_sysmat}. Now \Cref{thm:rootfinding} implies that $\{\lambda_1^{(j)}, \dots, \lambda_j^{(j)}\} = \Lambda(\tilde W^*A^*Z_j)$ are the roots of $\p{p}_j$, concluding the proof. \customqed
\end{proof}
We note that the latter result is the only one in this section for which there is no counterpart in the block case as the elements of block Krylov subspaces \eqref{eq:bkrylov} correspond to matrix polynomials, which in general can not be characterized by scalar roots.

We proceed with a theorem which connects the poles of the rational Krylov subspace and the eigenvalues of the projection of the matrix $A^*-X_jBB^*$. It is a generalization of parts of \cite[Thm. 4.4]{simoncini15} to oblique projections.
\begin{theorem}
With the same assumptions as in \Cref{lem:proj_sysmat}
	\begin{align}
		\Lambda(\tilde W^{*}(A^*-X_jBB^*)Z_j) = -\overline{s}
	\end{align}
	holds.
\end{theorem}
\begin{proof}
	Due to \Cref{def:RKD} and \Cref{thm:RADexists} the eigenvalues of $\underline{H_{-j}}$ are equal to the poles $s$ and so with \Cref{lem:proj_sysmat} and due to similarity
	\begin{align}
		\Lambda(\tilde W^*(A^*- X_jBB^*)Z_j) &= \Lambda(-Y_j\underline{H_{-j}}^*Y_j^{-1})\\
		&= \Lambda(-\underline{H_{-j}}^*)=-\overline{s}
	\end{align}
	holds, which concludes the proof. \customqed
\end{proof}

\subsection{The general case of a matrix $C^*\in \mathbb{C}^{n\times p}$}\label{sec:blockC}
So far we considered the case where $C^*$ is a vector. We now comment on the general case having a matrix $C^*\in \mathbb{C}^{n\times p}$ with an arbitrary $p\in \mathbb{N}$. We aim to construct an approximation $X_j=Z_jY_jZ_j^{*}$ with a regular $Y_j$ to the solution of the Riccati equation \eqref{def:are} with a rank-$p$ residual $\mathcal{R}(X_j)=R_jR_j^{*}$. The residual factor $R_j\in \mathbb{C}^{n\times p}$ then lies in a block rational Krylov subspace \eqref{eq:bkrylov} instead of a Krylov subspace \eqref{eq:ratkry}. To obtain a rank-$p$ residual we proceed in full analogy to the rank-1 residual by essentially replacing vectors by block vectors (i.e. matrices with $p$ columns) and scalars by scalar $p$-by-$p$ matrices (i.e. multiples of the $p$-dimensional identity matrix). We briefly comment on the changes to be made to obtain the sought-after approximation. Let a block rational Krylov subspace $\mathcal{K}^{\square}_j(A^*,C^*,s)$ be given with shifts $s=\left\{ \mu_1, \dots, \mu_j \right\} \subset \mathbb{C}$ fulfilling $s\cap -\overline s=\emptyset$. Let 
\begin{align}
	A^*V_{j+1}\underline{K_{j}} = V_{j+1}\underline{H_{j}}
\end{align}
be an associated BRAD as in \Cref{def:BRAD} with full-rank $V_{j+1}=\begin{bmatrix}C^*& Z_j\end{bmatrix}\in \mathbb{C}^{n\times jp}$ and block upper Hessenberg matrices
\begin{align}
	\underline{K_{j}} &= \begin{bmatrix}0\\ I\end{bmatrix}\in \mathbb{R}^{(j+1)p\times jp}\\
	\text{and } \underline{H_{j}} &= \begin{bmatrix}h_{j}\\ \underline{H_{-j}} \end{bmatrix}\in \mathbb{C}^{(j+1)p\times jp}, \text{ with } h_{j}\in \mathbb{C}^{p\times jp},\, \underline{H_{-j}}\in \mathbb{C}^{jp\times jp}
\end{align}
in analogy to \eqref{eq:RAD_KI}. In accordance to property \ref{it:BRADab} of \Cref{def:BRAD} the relation $H_{i+1,i}=\mu_iI_{p}$ must hold for the subdiagonal blocks of $\underline{H_{j}}$, which are the diagonal blocks of its quadratic lower submatrix $\underline{H_{-j}}$. Such a decomposition is constructed iteratively in \Cref{sec:iterexpa}. With the above BRAD \eqref{eq:VresV} holds with the block vector $v=\begin{bmatrix}I_{p}\\ 0\end{bmatrix}$ and \eqref{eq:sslyap} can be derived in full analogy. The resulting approximation $X_j=Z_jY_jZ_j^{*}$ yields a rank-$p$ residual.

\section{Two new iterative Riccati ADI methods}\label{sec:iterexpa}
Summarizing the approach from \Cref{sec:approx_riccati} the ADI approximate solution $X_j=Z_jY_jZ_j^*$ of the Riccati equation \eqref{def:are} is obtained in two steps:
First, generate a Krylov basis $Z_j$ with a corresponding RAD, then solve a small-scale Lyapunov equation for $Y_j^{-1}$. In this section these two steps are combined in an iterative way. The basis of the Krylov subspace is expanded incrementally, simultaneously the solution of the small-scale Lyapunov equation is updated.
Due to the importance of the Lyapunov equation \eqref{eq:sslyap} we give the following definition.
\begin{definition}
	For a relation
	\begin{align}\label{eq:specialRADp}
		A^*\begin{bmatrix}C^*& Z_j\end{bmatrix}\begin{bmatrix}0\\ I\end{bmatrix} = \begin{bmatrix}C^*& Z_j\end{bmatrix}\begin{bmatrix}h_{j}\\ \underline{H_{-j}} \end{bmatrix}
	\end{align}
	the \emph{associated $\tilde Y$-equation} is defined as
	\begin{align}\label{eq:yeq}
		0 &= \tilde Y\underline{H_{-j}}+ \underline{H_{-j}}^*\tilde Y - (S_j + h_j^*h_j)
	\end{align}
	with $S_j=Z_j^*BB^*Z_j$.
\end{definition}

In the remainder of this work we assume that all shifts have positive real part, i.e. $s\subset \mathbb{C}_{+}$. Hereby it is implied that the shift condition $s\cap -\overline s=\emptyset,$ $\infty\not\in s$ from \Cref{thm:exiuni} is satisfied. Further, in the $\tilde Y$-equation all eigenvalues of $\underline{H_{-j}}$ have positive real part as they are equal to the shifts and $(\underline{H_{-j}}^*,h_j^*)$ is controllable (cf. the paragraph preceding \Cref{thm:exiuni}). Due to \cite[Thm. 5.3.1 (b)]{lancaster95} the solution $\tilde Y$ is positive definite, and so is its inverse $Y_j=\tilde Y^{-1}$. Thus with a Cholesky decomposition of the small matrix $Y_j$, we find a $\hat Z_j$ spanning the same Krylov space as $Z_j$ with $X_j=Z_jY_jZ_j^*=\hat Z_j \hat Z_j^*$.

The (B)RAD \eqref{eq:specialRADp} corresponding to $Z_j$ can be altered to obtain a (B)RAD associated to $\hat Z_j$. To our advantage the corresponding $\tilde Y$-equation is then solved by the identity. The next lemma describes this procedure in detail and is the foundation of our new iterative ADI-type approaches to construct a rank-$p$ residual solution to \eqref{def:are}.
\begin{lemma}\label{lem:BRADtrafo}
	Let a BRAD of the form \eqref{eq:specialRADp} be given. Let all shifts $s\subset \mathbb{C}_{+}$ have positive real part such that the associated $\tilde Y$-equation has a positive definite solution $\tilde Y_j$. Let $G^*G = \tilde Y_j$ be a Cholesky decomposition of $\tilde Y_j$ with upper triangular $G$.
	Then the approximate solution $X_j=Z_j\tilde Y_j^{-1}Z_j^{*}$ yields a rank-$p$ residual $\mathcal{R}(X_j) = R_j R_j^*$ with residual factor $R_j = V_{j+1}\begin{bmatrix} I_p\\ \tilde Y_j^{-1}h_j^*\end{bmatrix}\in \mathbb{C}^{n\times p}$.
	This approximation can be transformed to $X_j=\hat Z_j \hat Z_j^{*}$ with $\hat Z_j=Z_jG^{-1}$.
Consider the associated transformed BRAD
	$A^* \hat V_{j+1}\begin{bmatrix}0\\ I\end{bmatrix} = \hat V_{j+1}\underline{{\hat H}_{j}}$
with $\hat V_{j+1}=\begin{bmatrix}C^*& \hat Z_j\end{bmatrix}=\begin{bmatrix}C^*& Z_j\end{bmatrix}\begin{bmatrix}I_p&0\\ 0& G^{-1}\end{bmatrix}$ and $\underline{\hat H_j}=\begin{bmatrix}I_p&0\\ 0& G\end{bmatrix}\underline{H_j}G^{-1}$. It has an associated $\tilde Y$-equation which is solved by the identity matrix, i.e.
\begin{align}\label{eq:lyapI}
	0 &= \underline{\hat H_{-j}}+ \underline{\hat H_{-j}}^* - (\hat S_j + \hat h_j^*\hat h_j)
\end{align}
holds with $\underline{\hat H_j}=\begin{bmatrix}\hat h_j\\ \underline{\hat H_{-j}}\end{bmatrix}$ and $\hat S_j=\hat Z_j^*BB^*\hat Z_j = G^{-*}S_jG^{-1}$. With these settings the residual factor can be expressed as $R_j = \hat V_{j+1}\begin{bmatrix} I_p\\ \hat h_j^*\end{bmatrix}\in \mathbb{C}^{n\times p}$.
\end{lemma}

In our iterative procedures after each expansion of the Krylov basis the corresponding BRAD is updated as described in \Cref{lem:BRADtrafo}. Consider a Krylov subspace $\mathcal{K}_j(A^*,C^*,s)$ with an associated BRAD
\begin{align}\label{eq:RADj}
	A^* \begin{bmatrix}C^*& Z_j\end{bmatrix} \begin{bmatrix}0\\ I\end{bmatrix} = \begin{bmatrix}C^*& Z_j\end{bmatrix}\underbrace{\begin{bmatrix}h_j\\ \underline{H_{-j}}\end{bmatrix}}_{= \underline{H_j}}
\end{align}
where the associated $\tilde Y$-equation is solved by the identity i.e. the low-rank residual approximation is given by $X_j=Z_jZ_j^*$ and the residual factor by
\begin{align}
R_j=\begin{bmatrix}C^*& Z_j\end{bmatrix} \begin{bmatrix}I_p\\ h_j^*\end{bmatrix}.
	\label{eq:RjI}
\end{align}
We expand the Krylov subspace $\mathcal{K}_j(A^*,C^*,s)$ by adding new poles, resulting in the space $\mathcal{K}_{\tilde{j}}(A^*,C^*,\tilde s)$ with $s\subset \tilde s$. The basis $\begin{bmatrix}C^*& Z_j\end{bmatrix}$ is expanded by a suitable $\tilde Z$ to obtain a basis of the expanded space. The BRAD corresponding to the expanded space $\mathcal{K}_{\tilde{j}}(A^*,C^*,\tilde s)$ thus reads
\begin{align}\label{eq:RADZ}
	A^* \begin{bmatrix}C^*& Z_j & \tilde Z\end{bmatrix} \begin{bmatrix}0\\ I\end{bmatrix} = \begin{bmatrix}C^*& Z_j & \tilde Z\end{bmatrix}\underline{\tilde H_{j+1}}.
\end{align}
The calculation of the expansion matrix $\tilde Z$ corresponding to a new pole $\mu$ is described in detail in the next two subsections. It is obtained through the solution of linear systems with the residual factor on the right hand side. By using different types of system matrices we obtain two iterative procedures: In the \emph{Riccati RAD iteration} linear systems with matrices of the form $A^*-\mu I_n$ are solved, while in the \emph{Lyapunov RADI iteration} systems are solved with matrices of the form $A^*-X_jBB^*-\mu I_n$. In \Cref{sec:prRADexp} we describe how the involved quantities can be kept real in case of real system matrices and complex shifts and how we can add multiple poles by making use of the solution of independent linear systems, which can be solved in parallel.

\subsection{The Riccati RAD iteration}\label{sec:R2ADi}
For the Riccati RAD iteration (R\textsuperscript{2}ADi) we expand the BRAD \eqref{eq:RADj} with $W_{j+1} = (A^*-\mu I_n)^{-1}R_j$. Rewriting this term we obtain equivalently
\begin{align}\label{eq:RADexp1}
	A^*W_{j+1} = R_j + \mu W_{j+1} = \begin{bmatrix}C^*& Z_j & W_{j+1}\end{bmatrix} \begin{bmatrix}I_p\\ h_j^*\\ \mu I_p\end{bmatrix}
\end{align}
with the representation of the residual factor $R_j$ as in \eqref{eq:RjI}.
Set $U_1=I_p,$ $U_2=h_j^{*}$ and $D=\mu I_p$. Then \eqref{eq:RADZ} holds with $\tilde Z = W_{j+1}$ and the partitioned matrix
\begin{align}\label{eq:extH}
	\underline{\tilde H_{j+1}} = \begin{bmatrix}h_j & U_1\\ \underline{H_{-j}} & U_2\\ 0 & D \end{bmatrix}.
\end{align}
We use the variables $U_1,$ $U_2$ with $U_2=h_j^{*}U_1$ and $D$ here to keep the derivation of our iterative procedure as general as necessary for the introduction of realification and parallelization later on. Due to $D=\mu I_p$ we find from \Cref{def:BRAD} that the BRAD \eqref{eq:RADj} is expanded with the pole $\mu$.

Our goal is to manipulate the expanded BRAD \eqref{eq:RADZ} so that the associated $\tilde Y$-equation is again solved by the identity matrix. In order to do so, first we set $\tilde h \coloneqq \begin{bmatrix}h_j& U_1\end{bmatrix},$ $\tilde H_{-} \coloneqq \begin{bmatrix}\underline{H_{-j}}& U_2\\ 0& D\end{bmatrix}$ and write the $\tilde Y$-equation associated to \eqref{eq:RADZ} as
\begin{align}\label{eq:ytlyap}
	\begin{split}
	0&=\tilde{Y}\tilde{H}_{-} + \tilde{H}_{-}^*\tilde{Y} - \begin{bmatrix}Z_j& \tilde Z\end{bmatrix}^*BB^*\begin{bmatrix}Z_j& \tilde Z\end{bmatrix} - \tilde{h}^*\tilde{h}\\
	&= \begin{bmatrix}Y_{11}\underline{H_{-j}} & Y_{11}U_2+Y_{12}D\\ Y_{12}^*\underline{H_{-j}} & Y_{12}^*U_2+Y_{22}D\end{bmatrix} + \begin{bmatrix}\underline{H_{-j}}^*Y_{11} & \underline{H_{-j}}^*Y_{12}\\ U_2^*Y_{11} + D^*Y_{12}^*&  U_2^*Y_{12}+D^*Y_{22} \end{bmatrix}\\
	&\phantom{=}- \begin{bmatrix}Z_j^*BB^*Z_j& Z_j^*BB^*\tilde Z \\ \tilde Z^* BB^*Z_j &\tilde Z^* BB^*\tilde Z\end{bmatrix} - \begin{bmatrix}h_j^*h_j& h_j^*U_1\\ U_1^*h_j& U_1^*U_1\end{bmatrix},
\end{split}
\end{align}
with the partitioned Hermitian solution matrix $\tilde Y=\begin{bmatrix} Y_{11}& Y_{12}\\ Y_{12}^*& Y_{22} \end{bmatrix}$. Now we examine the blocks of \eqref{eq:ytlyap} separately.
The upper left block of \eqref{eq:ytlyap} yields
\begin{align}
	0=Y_{11}\underline{H_{-j}} + \underline{H_{-j}}^*Y_{11} - Z_j^*BB^*Z_j - h_j^*h_j
\end{align}
with the solution $Y_{11}=I$ because this is the $\tilde Y$-equation associated to \eqref{eq:RADj} which was assumed to be solved by the identity. The lower left block is equal to the conjugated transposed of the upper right block. From the upper right block for $Y_{12}$ we obtain the Sylvester equation
\begin{align}\label{eq:Y12sylvester}
	\begin{split}
	0 &= U_2+Y_{12}D + \underline{H_{-j}}^*Y_{12} - Z_{j}^*BB^*\tilde Z - h_j^*U_1\\
	&= Y_{12}D + \underline{H_{-j}}^*Y_{12} - Z_{j}^*BB^*\tilde Z
\end{split}
\end{align}
due to $U_2=h_j^{*}U_1$. From the lower right block we find that $Y_{22}$ is the solution of the Lyapunov equation
\begin{align}\label{eq:Y22lyapunov}
	0 &= Y_{12}^*U_2 + Y_{22}D + U_2^*Y_{12} + D^*Y_{22} -\tilde Z^*BB^*\tilde Z - U_1^*U_1.
\end{align}
Thus, choosing $Y_{11}=I$ and solving \eqref{eq:Y12sylvester} for $Y_{12}$ and \eqref{eq:Y22lyapunov} for $Y_{22}$ yields the matrix $\tilde Y$ which solves \eqref{eq:ytlyap}. This $\tilde Y$ is used next to transform the BRAD \eqref{eq:RADZ} as described in \Cref{lem:BRADtrafo}. Only the newly added parts in the BRAD \eqref{eq:RADZ} are affected by this procedure, namely $\tilde Z,$ $U_1,$ $U_2,$ and $D$.

For the transformation let $G_{22}^*G_{22}=Y_{22}-Y_{12}^*Y_{12}$ be a Cholesky decomposition. Then
\begin{align}
	G = \begin{bmatrix} I & Y_{12}\\ 0& G_{22} \end{bmatrix} 
\end{align}
is the Cholesky factor of $\tilde Y=G^*G$. Now postmultiply the BRAD \eqref{eq:RADZ} by 
\begin{align}
	G^{-1} = \begin{bmatrix} I & -Y_{12}G_{22}^{-1}\\ 0& G_{22}^{-1} \end{bmatrix}
\end{align}
and rewrite the resulting equation as 
\begin{align}\label{eq:RADZjp1}
	A^* \begin{bmatrix}C^*& Z_{j+1} \end{bmatrix} \begin{bmatrix}0\\ I\end{bmatrix} = \begin{bmatrix}C^*& Z_{j+1}\end{bmatrix}\underline{H_{j+1}}
\end{align}
with
\begin{align}\label{eq:Htrafo}
	\underline{H_{j+1}} & =\begin{bmatrix}1&0\\ 0& G\end{bmatrix} \underline{\tilde H_{j+1}}G^{-1}= \begin{bmatrix}h_j & \hat U_1\\ \underline{H_{-j}} & \hat U_2\\ 0 & \hat D \end{bmatrix}
\end{align}
where
\begin{align}
	\hat U_1 &= \left(-h_jY_{12} + U_1\right )G_{22}^{-1}\\
	\hat U_2 &= \big(-\underline{H_{-j}}Y_{12} + U_2 + Y_{12}D\big)G_{22}^{-1}\\
	\hat D &= G_{22}DG_{22}^{-1}
\end{align}
and with
\begin{align}\label{eq:Ztrafo}
	Z_{j+1}&=\begin{bmatrix} Z_j & \tilde Z\end{bmatrix}G^{-1} = \begin{bmatrix} Z_j & \hat Z\end{bmatrix}
\end{align}
where $\hat Z = \left(-Z_jY_{12}+ \tilde Z \right)G_{22}^{-1}$. Due to this manipulation of the BRAD \eqref{eq:RADZ} now the $\tilde Y$-equation associated to the transformed BRAD \eqref{eq:RADZjp1} is solved by the identity matrix, i.e. \eqref{eq:lyapI} is satisfied again.
For the residual factor we find
\begin{align}
	R_{j+1} &= \begin{bmatrix}C^*& Z_{j+1}\end{bmatrix}\begin{bmatrix}I_p \\ h_{j+1}^*\end{bmatrix} = \begin{bmatrix}C^*& Z_j & \hat Z \end{bmatrix}\begin{bmatrix}I_p \\ h_j^* \\ \hat U_1^*\end{bmatrix} = R_j + \hat Z \hat U_1^*.
\end{align}
This iterative procedure is summarized in \Cref{alg:R2ADi}.

\begin{algorithm}[t]
	\caption{Riccati RAD iteration (R\textsuperscript{2}ADi)}
    \label{alg:R2ADi}
    \begin{algorithmic}[1]
    \Input System matrices $A,$ $B,$ $C,$ set of shifts $s\subset \mathbb{C}_{+}$
    \Output approximate solution $Z Z^*$, residual factor $R$
    \State initialize $Z_0=[\,],$ $R_0=C^*,$ $h_0=[\,],$ $\underline{H_{-0}}=[\,],$ $s_0=0,$ $j=0$
    \While {not converged}
    \State obtain new shift(s) $\mu$ from $s$
    \State expand RAD, obtain $\tilde Z,$ $U_1,$ $U_2,$ $D$ \Comment{solve $(A^*-\mu I_n)^{-1}R_j$}
    \State solve $Y_{12}D + \underline{H_{-j}}^*Y_{12} - s_{j}^*(B^*\tilde Z) = 0$ for $Y_{12}$\label{lin:sylvesterY12}
    \State solve $Y_{12}^*U_2 + Y_{22}D + U_2^*Y_{12} + D^*Y_{22} -\tilde Z^*BB^*\tilde Z - U_1^*U_1 = 0$ for $Y_{22}$\label{lin:lyapY22}
    \State compute $G_{22}=\chol(Y_{22}-Y_{12}^*Y_{12})$
    \State compute $[\hat U_1,\, \hat U_2,\, \hat D] = [-h_{j}Y_{12} + U_1,\, -\underline{H_{-j}}Y_{12} + U_2 + Y_{12}D,\, G_{22}D]G_{22}^{-1}$
    \State compute $\hat Z = (-Z_{j}Y_{12}+ \tilde Z)G_{22}^{-1}$ \label{lin:ZY12}
    \State update $Z_{j+1}=[Z_{j},\, \hat Z],$ $R_{j+1}=R_{j}+\hat Z\hat U_1^{*}$ \label{lin:Rupdate}
    \State update $h_{j+1}=[h_{j}, \hat U_1],$ $\underline{H_{-(j+1)}}=\begin{bmatrix}\underline{H_{-j}}& \hat U_2\\ 0 &\hat D\end{bmatrix}$
    \State update $s_{j+1} = [s_{j},\, B^*\hat Z]$
    \State $j=j+1$
    \EndWhile
    \State $Z=Z_j,$ $R=R_j$
    \end{algorithmic}
\end{algorithm}

\subsection{The Lyapunov RADI iteration}\label{sec:radi_iteration}
We now consider the expansion of the BRAD \eqref{eq:RADj} with $W_{j+1}=(A^*-X_jBB^*-\mu I_n)^{-1}R_j$. We find equivalently
\begin{align}
	A^*W_{j+1} &= R_j + X_jBB^*W_{j+1} + \mu W_{j+1}\\
	&= \begin{bmatrix}C^*& Z_j\end{bmatrix}\begin{bmatrix}I_p\\ h_j^{*}\end{bmatrix} + Z_j(Z_j^{*}BB^*W_{j+1}) + \mu W_{j+1}\\
	&= \begin{bmatrix}C^*& Z_j & \tilde Z\end{bmatrix}\begin{bmatrix}I_p\\ h_j^{*} + Z_j^{*}BB^*\tilde Z \\ \mu I_p \end{bmatrix}
\end{align}
due to the representation of the residual factor $R_j$ as in \eqref{eq:RjI} and $X_j=Z_jZ_j^{*}$. Thus \eqref{eq:RADZ} holds with $\tilde Z = W_{j+1}$ and
\begin{align}\label{eq:extH2}
	\underline{\tilde H_{j+1}} = \begin{bmatrix}h_j & U_1\\ \underline{H_{-j}} & U_2 + Z_j^{*}BB^*\tilde Z\\ 0 & D \end{bmatrix},
\end{align}
where $U_1=I_p,$ $U_2=h_j^*,$ and $D=\mu I_p$ holds. Again we use the generic variables $U_1,$ $U_2$ with $U_2=h_j^{*}U_1$ and $D$ in preparation for realification and parallelization in \Cref{sec:prRADexp}. As in the previous subsection $D=\mu I_p$ holds, so the BRAD \eqref{eq:RADj} is expanded with the pole $\mu$ here, too.

We proceed as in \Cref{sec:R2ADi} and examine each block of the partitioned $\tilde Y$-equation associated to the BRAD \eqref{eq:RADZ} in analogy to \eqref{eq:ytlyap}. Due to the additional term $Z_j^{*}BB^*\tilde Z$ in \eqref{eq:extH2} we have to replace every occurrence of $U_2$ in the equations of \Cref{sec:R2ADi} with $U_2+Z_j^{*}BB^*\tilde Z$. Thus we find that in \eqref{eq:Y12sylvester} the term $Z_j^{*}BB^*\tilde Z$ cancels out, resulting in the equation
\begin{align}\label{eq:Y12sylvester0}
	0 = Y_{12}D + \underline{H_{-j}}^{*}Y_{12}
\end{align}
which is solved by $Y_{12}=0$. By inserting $Y_{12}=0$ into \eqref{eq:Y22lyapunov} we obtain $Y_{22}$ as the solution of the Lyapunov equation
\begin{align}\label{eq:Y22RADI}
	0 &= Y_{22}D + D^*Y_{22} -\tilde Z^*BB^*\tilde Z - U_1^*U_1.
\end{align}
As at the end of \Cref{sec:R2ADi} but now with $Y_{12}=0$ we obtain
\begin{align}
	A^*\begin{bmatrix}C^*& Z_{j+1}\end{bmatrix} \begin{bmatrix}0\\ I\end{bmatrix} = \begin{bmatrix}C^*& Z_{j+1}\end{bmatrix}\underline{H_{j+1}}
\end{align}
with $Z_{j+1}=\begin{bmatrix}Z_{j} & \tilde Z G_{22}^{-1}\end{bmatrix}$ and
\begin{align}\label{eq:HtrafoRADI}
	\underline{H_{j+1}} & =\begin{bmatrix}1&0\\ 0& G\end{bmatrix} \underline{\tilde H_{j+1}}G^{-1}= \begin{bmatrix}h_j & U_1G_{22}^{-1}\\ \underline{H_{-j}} & (U_2+Z_j^{*}BB^*\tilde Z)G_{22}^{-1}\\ 0 & G_{22}DG_{22}^{-1} \end{bmatrix},
\end{align}
with $G=\diag(I,G_{22})$ and the Cholesky decomposition $G_{22}^*G_{22}=Y_{22}$. For the residual factor
\begin{align}
	R_{j+1} &= R_{j} + (\tilde Z G_{22}^{-1})(U_1G_{22}^{-1})^*\\
	&= R_{j} + \tilde Z Y_{22}^{-1} U_1^*
\end{align}
holds. We summarize this iteration in \Cref{alg:lyapRADI}. 
\begin{algorithm}[t]
	\caption{Lyapunov RADI iteration}
    \label{alg:lyapRADI}
    \begin{algorithmic}[1] 
    \Input System matrices $A,$ $B,$ $C,$ set of shifts $s\subset \mathbb{C}_{+}$
    \Output approximate solution $ZZ^*$, residual factor $R$
    \State initialize $Z_0=[\,],$ $R_0=C^*,$ $h_0=[\,],$ $\underline{H_{-0}}=[\,],$ $s_0=0,$ $K_0=0,$ $j=0$
    \While {not converged}
    \State obtain new shift(s) $\mu$ from $s$
    \State expand RAD, obtain $\tilde Z,$ $U_1,$ $U_2,$ $D$ \Comment{solve $(A^*-K_{j}B^*-\mu I_n)^{-1}R_{j}$}
    \State solve $Y_{22}D + D^*Y_{22} -\tilde Z^*BB^*\tilde Z - U_1^*U_1 = 0$ for $Y_{22}$
    \State compute $G_{22}=\chol(Y_{22})$
    \State compute $[\hat U_1,\, \hat U_2,\, \hat D] = [U_1,\, U_2+s_{j}^*(B^*\tilde Z),\, G_{22}D]G_{22}^{-1}$\label{lin:U1U2D}
    \State compute $\hat Z = \tilde ZG_{22}^{-1}$
    \State update $Z_{j+1}=[Z_{j},\, \hat Z],$ $R_{j+1}=R_{j}+\hat Z\hat U_1^{*}$\label{lin:ZR}
    \State update $h_{j+1}=[h_{j}, \hat U_1],$ $\underline{H_{-(j+1)}}=\begin{bmatrix}\underline{H_{-j}}& \hat U_2\\ 0 &\hat D\end{bmatrix}$\label{lin:hH}
    \State update $s_{j+1} = [s_{j},\, B^*\hat Z]$\label{lin:s}
    \State update $K_{j+1} = K_{j} + \hat Z (\hat Z^* B)$ \label{lin:Kupdate}
    \State $j=j+1$
    \EndWhile
    \State $Z=Z_j,$ $R=R_j$
    \end{algorithmic}
\end{algorithm}

Note that to obtain $Z$ and $R$ in \Cref{alg:lyapRADI} it is not necessary to compute and store $\underline{H_{j}}$. When $\hat U_1^*$ in line \ref{lin:ZR} is replaced by $G_{22}^{-*}U_1^{*}$, the lines \ref{lin:U1U2D}, \ref{lin:hH}, \ref{lin:s} and the variables $h_j,$ $\underline{H_{-j}},$ $s_j$, $U_2$ can be deleted. Nevertheless, having $\underline{H_{j}}$ available might be useful in some situations and only neglectable amount of memory and computational time is necessary to create it.

\begin{remark}
	In general the matrix $A^*-K_jB^*-\mu I_n$ with the feedback term $K_j=X_jB$ is a dense matrix due to the term $K_jB^*$, although $A^*$ is sparse, making system solves costly. However, as $K_jB^*$ is of rank $m$, it was proposed in \cite[Sec. 4.2]{benner18} to use the Sherman-Morrison-Woodbury (SMW) formula to speed up computations. The formula reads
	\begin{align}
		(A^*-K_jB^*-\mu I_n)^{-1}R_j &= L + N (I_m -B^*N)^{-1}B^*L \label{eq:smw1}\\
		\begin{bmatrix}L&N\end{bmatrix} &= (A^*-\mu I_n)^{-1}\begin{bmatrix}R_{j}& K_{j}\end{bmatrix}.\label{eq:smw2}
	\end{align}
	To obtain the solution of \eqref{eq:smw1}, first the sparse linear system in \eqref{eq:smw2} is solved for $L$ and $N$, then the right hand side of \eqref{eq:smw1} is used.
\end{remark}

\begin{remark}
	The procedure derived in this subsection is essentially equivalent to the RADI iteration \cite{benner18} with mainly two differences. First, instead of an approximation $X_j=Z_jY_j^{-1}Z_j^{*}$ with a (block) diagonal matrix $Y_j$ we use a Cholesky factorization of $Y_j^{-1}$ to put it into the factor $Z_j$, resulting in the approximate solution $X_j=Z_jZ_j^{*}$. Second, the shift parameters $\alpha_j$ in the RADI iteration correspond to the negative poles of our BRADs, therefore the parameters must be chosen as $\mu = -\alpha_j$ to obtain an equivalent approximation.

We further note that for $U_1=I_p$ and $D=\mu$ as above \eqref{eq:Y22RADI} simplifies to
\begin{align}
	2\Real(\mu)Y_{22} = I_p + \tilde Z^* BB^* \tilde Z,
\end{align}
which is, up to constants, equivalent to line 10 of \cite[Alg. 1]{benner18}. However, we prefer to solve the more general Lyapunov equation \eqref{eq:Y22RADI} because it is more versatile: It allows for realified and parallel RAD expansions which is described in the next subsection. 

Due to the Lyapunov equation \eqref{eq:Y22RADI} and the equivalence to the RADI iteration we chose the name Lyapunov RADI iteration for \Cref{alg:lyapRADI}. 
\end{remark}

\begin{algorithm}[t]
    \caption{Simple RAD expansion}
    \label{alg:sRADexp}
    \begin{algorithmic}[1]
	    \Input residual factor $R_j=\begin{bmatrix}C^* & Z_j\end{bmatrix}\begin{bmatrix}I_p \\ h_j^*\end{bmatrix}$, shift $\mu\in \mathbb{C}\setminus\Lambda(A^*)$
	    \Output  $\tilde Z$, $U_1,$ $U_2,$ $D$
    \State solve $W = (A^*-\mu I_n)^{-1}R_j$ or $W = (A^*-K_jB^*-\mu I_n)^{-1}R_j$
    \State $\tilde Z = W$
    \State $U_1=I_p,$ $U_2=h_j^{*},$ $D=\mu I_p$
    \end{algorithmic}
\end{algorithm}

\subsection{Parallel and realified expansion of the Krylov basis}\label{sec:prRADexp}
In the preceding subsections we have expanded the BRAD \eqref{eq:RADj} with a new pole $\mu$ and obtained the BRAD \eqref{eq:RADZ} with
\begin{align}
	\underline{\tilde H_{j+1}} = \begin{bmatrix}h_j & U_1\\ \underline{H_{-j}} & U_2\\ 0 & D \end{bmatrix} \text{ or }
	\underline{\tilde H_{j+1}} = \begin{bmatrix}h_j & U_1\\ \underline{H_{-j}} & U_2 + Z_j^{*}BB^*\tilde Z\\ 0 & D \end{bmatrix}
\end{align}
with $D=\mu I_p$ as in \eqref{eq:extH} and \eqref{eq:extH2}. Only one linear system with a shift $\mu$ was solved, which is summarized in \Cref{alg:sRADexp}.

We now describe how the BRAD can be expanded in parallel in the R\textsuperscript{2}ADi. Hereby we mean expanding the Krylov basis with $W_i = (A^*-\mu_i I_n)^{-1}R_j$ for $i=1, \dots, l$ corresponding to the pairwise distinct shifts $\mu_1, \ldots, \mu_l$. These linear systems are independent from each other and can thus be solved in parallel. Rewriting these linear systems we obtain
\begin{align}
	A^*W_i = R_j + \mu_i W_i = \begin{bmatrix}C^*& Z_j & W_i\end{bmatrix}\begin{bmatrix}I_p \\ h_j^*\\ \mu_i I_p\end{bmatrix}
\end{align}
as in \eqref{eq:RADexp1}. Thus the expanded BRAD \eqref{eq:RADZ} with $\tilde Z=\begin{bmatrix}W_1 & \cdots & W_l\end{bmatrix}$ is given by
\begin{align}\label{eq:RADj2}
	A^* \begin{bmatrix}C^*& Z_j & \tilde Z\end{bmatrix} \begin{bmatrix}0\\ I\end{bmatrix} = \begin{bmatrix}C^*& Z_j & \tilde Z\end{bmatrix}\begin{bmatrix}h_j & I_p & \cdots & I_p\\ \underline{H_{-j}} & h_j^*&\cdots&h_j^*\\ 0 & \mu_1I_p & & \\ \vdots &&\ddots& \\ 0& &&\mu_lI_p\end{bmatrix}.
\end{align}
This procedure is summarized in terms of $U_1,$ $U_2$ with $U_2=h_j^*U_1$ and $D$ in \Cref{alg:pRADexp}. Due to \Cref{def:BRAD} and the structure of $D$ we find that the BRAD is expanded with the poles $\mu_1, \ldots, \mu_l.$

\begin{algorithm}[t]
    \caption{Parallel RAD expansion}
    \label{alg:pRADexp}
    \begin{algorithmic}[1]
	    \Input residual factor $R_j=\begin{bmatrix}C^* & Z_j\end{bmatrix}\begin{bmatrix}I_p \\ h_j^*\end{bmatrix}$, pairwise distinct shifts $\mu_1, \dots, \mu_l\in \mathbb{C}\setminus\Lambda(A^*)$
	    \Output  $\tilde Z$, $U_1,$ $U_2,$ $D$
    \State solve $W_i = (A^*-\mu_i I_n)^{-1}R_j$ or $W_i = (A^*-K_jB^*-\mu_i I_n)^{-1}R_j$ for $i=1, \dots, l$
    \State $\tilde Z = \left[W_1, \dots, W_l \right]$
    \State $U_1=\left[I_p, \dots, I_p\right],$ $U_2=\left[h_j^{*}, \dots, h_j^{*}\right],$ $D=\diag(\mu_1I_p, \dots, \mu_lI_p)$
    \end{algorithmic}
\end{algorithm}

Next we consider the question of how to modify the R\textsuperscript{2}ADi such that in case of real system matrices $A,$ $B,$ $C$ and two complex conjugated shifts the iterates remain real valued and how the use of complex arithmetic is minimized; in short, we consider realification. Let $R_j$ have only real entries and let the BRAD \eqref{eq:RADj} be expanded by the two conjugated shifts $\mu,\overline\mu\in \mathbb{C}\setminus \mathbb{R},$ $\mu=a+bi,$ $a,b\in \mathbb{R}$. Set $W=(A^*-\mu I_n)^{-1}R_j$ and so $\overline{W} = (A^*-\overline\mu I_n)^{-1}R_j$. With $S=\frac{1}{2}\begin{bmatrix}1& -i\\ 1 & i\end{bmatrix}\otimes I_p$ we see $\begin{bmatrix}W&\overline{W}\end{bmatrix}S=\begin{bmatrix}\Real(W)&\Imag(W)\end{bmatrix}$ and $S^{-1}\begin{bmatrix}\mu I_p&\\ &\overline{\mu}I_p\end{bmatrix}S=\begin{bmatrix}aI_p& bI_p\\ -bI_p&aI_p\end{bmatrix}$. Parallel expansion of the Krylov basis with the complex block vectors $W$ and $\overline{W}$ yields the expanded BRAD
\begin{align}
	A^* \begin{bmatrix}C^*& Z_j & W & \overline{W}\end{bmatrix} \begin{bmatrix}0\\ I\end{bmatrix} &= \begin{bmatrix}C^*& Z_j & W&\overline{W}\end{bmatrix}\begin{bmatrix}h_j & I_p & I_p\\ \underline{H_{-j}} & h_j^*&h_j^*\\ 0 & \mu I_p &\\ 0&&\overline{\mu}I_p\end{bmatrix}.
\end{align}
Transformation of this BRAD with the matrix $\begin{bmatrix}I& 0\\ 0& S\end{bmatrix}$ then yields the equivalent realified BRAD
\begin{align}
	A^* \Big[C^*\ Z_j \ \Real(W) \ \Imag(W)\Big] \begin{bmatrix}0\\ I\end{bmatrix} &= \Big [ C^*\ Z_j \ \Real(W) \ \Imag(W)\Big] \begin{bmatrix}h_j & I_p & 0\\ \underline{H_{-j}} & h_j^*&0\\ 0 & aI_p &bI_p\\0&-bI_p&aI_p\end{bmatrix},\label{eq:RADreal}
\end{align}
where the Krylov basis is expanded with the real block vectors $\Real(W)$ and $\Imag(W)$. Note that only one (complex) system solve is necessary for the expansion with the two shifts $\mu$ and $\overline\mu$.
The realified expansion is stated in terms of $U_1,$ $U_2$ with $U_2=h_j^*U_1$ and $D$ in \Cref{alg:rRADexp}.

\begin{algorithm}[t]
    \caption{Realified RAD expansion}
    \label{alg:rRADexp}
    \begin{algorithmic}[1] 
	    \Input real residual factor $R_j=\begin{bmatrix}C^* & Z_j\end{bmatrix}\begin{bmatrix}I_p \\ h_j^*\end{bmatrix}$, shifts $a+bi=\mu, \overline\mu \in \mathbb{C}\setminus\Lambda(A^*)$ with $b\neq 0$
	    \Output  $\tilde Z$, $U_1,$ $U_2,$ $D$
    \State solve $W = (A^*-\mu I_n)^{-1}R_j$ or $W = (A^*-K_jB^*-\mu I_n)^{-1}R_j$
    \State $\tilde Z = \left[\Real(W), \Imag(W) \right]$
    \State $U_1=\left[I_p, 0\right],$ $U_2=\left[h_j^{*}, 0\right],$ $D=\begin{bmatrix}a&b\\-b&a\end{bmatrix}\otimes I_p$
    \end{algorithmic}
\end{algorithm}

\begin{remark}
	Unfortunately after the realification the relation \eqref{eq:RADreal} does not fulfill property \ref{it:BRADab} of \Cref{def:BRAD}. It therefore does not satisfy the definition of a BRAD anymore. We thus propose the term \emph{quasi BRAD} in analogy to the term quasi RAD. Further, for computational reasons it is beneficial to permute the columns of $\Real(W)$ and $\Imag(W)$ so that the Krylov basis is expanded by $\begin{bmatrix}\Real(w_1)& \Imag(w_1)& \cdots& \Real(w_p)& \Imag(w_p)\end{bmatrix}$ for $W=\begin{bmatrix}w_1& \cdots& w_p\end{bmatrix}$ as then the lower right block $\begin{bmatrix}a & b\\-b & a\end{bmatrix}\otimes I_p$ becomes $I_p \otimes \begin{bmatrix}a& b\\ -b & a\end{bmatrix}$, a block diagonal matrix with $2$-by-$2$ blocks on the diagonal. This permutation makes $\underline{H_{-j}}$ in the associated $\tilde Y$-equation a quasi upper triangular matrix so the Sylvester equation in line \ref{lin:sylvesterY12} of \Cref{alg:R2ADi} can be solved efficiently with established software packages.
\end{remark}

The parallel and realified expansion was derived here only for the Riccati RAD iteration. It can be adapted for the Lyapunov RADI iteration by simply incorporating the term $K_jB^*$ in the linear systems.

\section{Discussion}\label{sec:discussion}
As seen in \Cref{sec:approx_riccati} the two key assumptions for the existence of the Riccati ADI solution were that the solution factor $Z$ in the approximate solution is a basis of a certain (block) Krylov subspace and that the Riccati residual is of rank $p$. Unfortunately, in case of deflation or the lucky case of an exact solution, the residual is of rank smaller than $p$. To avoid these special cases the rank-$p$ condition can be replaced with the requirement that the Riccati residual factor is an element of an augmented (block) Krylov subspace which by definition contains (block) vectors with $p$ columns. However, we did not observe deflation in practice and therefore employed the rank-$p$ condition throughout this work.

We have already stated in the introduction that there are several Riccati ADI methods which have been shown to be equivalent e.g. in \cite{benbu16,benner18} by explicitly proving the equivalence of the algorithms. The uniqueness result \Cref{thm:exiuni} makes it simpler to proof the equivalence of different Riccati ADI methods. One essentially has to verify that the approximate solution lies in a rational Krylov subspace with a set of poles $s\subset \mathbb{C}$ satisfying $s\cap -\overline{s}=\emptyset$ and that the Riccati residual is of rank $p$.

\subsection{Generalized Riccati equations}\label{sec:genriccati}
We consider two generalizations of the Riccati equation \eqref{def:are}. We first describe how the generalized Riccati equation
\begin{align}
	\label{eq:areE}
	A^{*}XE + E^{*}XA + C^*C - E^{*}XBB^{*}XE = 0
\end{align}
with an additional regular system matrix $E\in \mathbb{C}^{n\times n}$ affects our iteration. This equation is not solved directly. Instead, as in \cite[Sec. 4.4]{benner18}, the equivalent Riccati equation
\begin{align}
	\label{eq:areEA}
	E^{-*}A^{*}X + XAE^{-1} + E^{-*}C^*CE^{-1} - XBB^{*}X = 0
\end{align}
is considered. It has the same structure as \eqref{def:are} where the system matrices $A$ and the initial residual factor $C^*$ are replaced by $AE^{-1}$ and $E^{-*}C^*$. It can therefore be solved with the methods described in the preceding sections. In an efficient iteration inverting $E$ is avoided by utilizing the relation
\begin{align}
	((AE^{-1})^{*}-\mu I_n)^{-1}E^{-*}R = (A^*-\mu E^{*})^{-1}R
\end{align}
with a residual factor $R$ of \eqref{eq:areE}. This requires the following modifications in the algorithms presented. All linear systems have to be shifted by a multiple of $E^*$ instead of $I_n$, i.e. the systems $(A^*-\mu E^*)^{-1}R_{j}$ respectively $(A^*-K_{j}B^*-\mu E^*)^{-1}R_{j}$ have to be solved. Further, the residual and feedback updates have to be multiplied with $E^*$ to convert these quantities corresponding to \eqref{eq:areEA} into ones corresponding to \eqref{eq:areE}. For instance, the residual update in line \ref{lin:Rupdate} of \Cref{alg:R2ADi} has to be replaced by $R_{j+1}=R_{j}+E^{*}\hat Z\hat U_1^{*}$ and the feedback update in line \ref{lin:Kupdate} of \Cref{alg:lyapRADI} must be replaced by $K_{j+1} = K_{j} + E^* \hat Z (\hat Z^* B)$. In all other algorithms these lines have to be modified alike.

Our approach can also be applied to the nonsymmetric Riccati equation
\begin{align}
	A_1^{*}X + XA_2 + C_1^*C_2 - XB_2B_1^{*}X = 0
	\label{eq:arenonsym}
\end{align}
with $A_i\in \mathbb{C}^{n_i\times n_i},$ $B_i\in \mathbb{C}^{n_i\times m},$ and $C_i\in \mathbb{C}^{p\times n_i}$ for $i=1,2$ and $X\in \mathbb{C}^{n_1\times n_2}$. It can be solved in analogy to the symmetric Riccati equation as follows. Consider the two decompositions
\begin{align}
	A_i^{*}\begin{bmatrix}C_i^*& Z_{i}\end{bmatrix}\underline{K_{i}}=\begin{bmatrix}C_{i}^*& Z_{i}\end{bmatrix}\underline{H_{i}}\text{ for } i=1,2,
\end{align}
with $\underline{K_{i}}=\begin{bmatrix}0\\I\end{bmatrix}$ and $\underline{H_{i}}=\begin{bmatrix}h_i\\\underline{H_{-i}}\end{bmatrix}$ similar to \Cref{lem:specialRAD} and with the same number of columns in $Z_1$ and $Z_2$. Then in analogy to \eqref{eq:VresV} we can rewrite the residual \eqref{eq:arenonsym} for the approximate solution $X=Z_1YZ_2^{*}$ as
\begin{align}
	\begin{bmatrix}C_1^*& Z_{1}\end{bmatrix}\left( \underline{H_1} Y \underline{K_2}^* + \underline{K_1} Y \underline{H_2}^* + \begin{bmatrix}I_p \\ 0\end{bmatrix}\begin{bmatrix}I_p & 0\end{bmatrix} - \underline{K_1} Y S Y \underline{K_2}^* \right)\begin{bmatrix}C_2^*& Z_{2}\end{bmatrix}^*
\end{align}
due to $Z_i=\begin{bmatrix}C_i^*& Z_{i}\end{bmatrix}\underline{K_{i}}$ and with $S=Z_2^*B_2B_1^*Z_1$. The associated $\tilde Y$-equation as in \eqref{eq:sslyap} which determines $\tilde Y=Y^{-1}$ so that a rank-$p$ residual is obtained has to be adapted accordingly and results in the Sylvester equation
\begin{align}\label{eq:YSylvester}
	0 &= \tilde Y\underline{H_{-1}}+ \underline{H_{-2}}^*\tilde Y - (S + h_2^*h_1)
\end{align}
and the approximate solution $X=Z_1YZ_2^{*}$ of \eqref{eq:arenonsym}. Of course also the derivation of the iterative procedure in \Cref{sec:iterexpa} can be transferred to the nonsymmetric Riccati equation. We briefly mention where such a procedure differs from \Cref{alg:R2ADi}. On the one hand, clearly two RADs have to be created instead of one, doubling the computational effort for the solution of linear systems. On the other hand, the associated $\tilde Y$-equation becomes the Sylvester equation \eqref{eq:YSylvester} and is in general solved by a non Hermitian matrix. Therefore the upper right and lower left block of the solution $\tilde Y$ of \eqref{eq:YSylvester} as in \eqref{eq:ytlyap} are not connected via conjugated transposition and two Sylvester equations like \eqref{eq:Y12sylvester} have to be solved. Further, the Lyapunov equation \eqref{eq:Y22lyapunov} becomes a Sylvester equation and yields a non Hermitian solution, such that no Cholesky decomposition exists. However, the matrix $Y$ can be decomposed in arbitrary ways, e.g. trivially into $Y=YI=IY$ resulting in $(Z_1Y)Z_2^{*}$ and $Z_1(YZ_2^{*})$. Another possible decomposition is a $LDU$ factorization with $L$ and $U$ being (block) triangular matrices with ones on the diagonal and a $D$ being a (block) diagonal matrix. For parallelization and realification no further changes are necessary as both happens in the algorithms for the RAD expansion.

\subsection{Linear matrix equations}
All results in this work immediately transfer to Lyapunov equations which are a special case of the Riccati equation \eqref{def:are} with $B=0$ and so $S_j=Z_j^*BB^*Z=0$. The R\textsuperscript{2}ADi and Lyapunov RADI iteration both simplify considerably. In the R\textsuperscript{2}ADi there is no need for the multiplication of $Y_{12}$ with the previous Krylov basis and in the Lyapunov RADI iteration there is no need for the SMW formula any more and both algorithms become equivalent.

It is notable that even for the linear Lyapunov equations the projected equation \eqref{eq:areP} respectively \eqref{eq:ssricc} is a Riccati equation, i.e. a quadratic matrix equation. This was also found in \cite[Sec. IV]{ahmad10} for large scale Sylvester equations. However, as this small scale Riccati equation is homogeneous, the equivalent small scale Lyapunov equation \eqref{eq:sslyap} can be solved for the inverse of the solution instead. The discussion in \Cref{sec:radi_iteration} implies that this solution is a (block) diagonal matrix.

Also the residual formula from \Cref{thm:Rrat} simplifies considerably. It becomes 
\begin{align}\label{eq:lyapresfac}
	R_j= \frac{\prod_{i=1}^j (A^*+\overline{s_i}I)}{\prod_{i=1}^j (A^*-s_iI)}C^*
\end{align}
as $\Lambda(\tilde W^*A^*Z_j)=\Lambda(-Y_j\underline{H_{-j}}^{*}Y_j^{-1})=\Lambda(-\underline{H_{-j}}^{*})=-\overline{s}$ holds due to \Cref{lem:proj_sysmat} and $B=0$. This means that the roots remain constant from step to step in the iteration, other than in the quadratic case.

\subsection{Shift selection}\label{sec:shifts}
For a good approximate solution the choice of the poles of the rational Krylov subspace used in the approximate solution $X_j$ is crucial. Many shift strategies exist, but their description is beyond the scope of this work. Here we only describe, in concise form and in our notation, the \emph{residual Hamiltonian shift strategy} from \cite[Sec. 4.5.1]{benner18} which we use in our numerical experiments. For a detailed discussion of this and other shift strategies we refer to \cite[Sec. 4.5]{benner18}.

The residual Hamiltonian shift strategy makes use of the eigenvalues of the Hamiltonian matrix
\begin{align}\label{eq:hamproj}
	\mathscr{H}^{\text{proj}}=\begin{bmatrix}U^*\tilde{A}U& U^*BB^*U\\ U^* R_j R_j^{*}U& -U^*\tilde{A}^*U\end{bmatrix}
\end{align}
where $\tilde{A}=A-BB^*X_j$ holds and $U$ is an orthonormal matrix which spans the same space as the last $l$ columns of $Z_j$, where $l$ is a parameter of choice. 
Let $\hat \lambda$ be an eigenvalue of $\mathscr{H}^{\text{proj}}$ with the corresponding eigenvector $\begin{bmatrix}\hat r\\ \hat q\end{bmatrix}$. The next shift is chosen as $\mu =-\hat \lambda$ where $\hat \lambda$ maximizes the expression 
$\lVert \hat q(\hat q^*\hat r)^{-1}\hat q^* \rVert$
as a heuristic for fast convergence.

\section{Numerical experiments}\label{sec:numexp}
We perform numerical experiments to compare the Riccati RAD iteration as in \Cref{alg:R2ADi} with the Lyapunov RADI iteration as in \Cref{alg:lyapRADI} and show the effects of our parallel approach. The algorithms were implemented including the modifications described in \Cref{sec:genriccati} to handle generalized Riccati equations \eqref{eq:areE} with $E$ and with realification and the possibility to use parallelization as described in \Cref{sec:prRADexp}. Our implementations of both algorithms share large parts of their code, which allows for a fair performance comparison. We compared the resulting approximate solutions of our implementations with the M-M.E.S.S.-2.0 \cite{mmess20} RADI implementation and only found differences in the order of the machine precision. Therefore we only state results for the R\textsuperscript{2}ADi and the Lyapunov RADI iteration.

\begin{figure}[t]
\def \rradirail{./plots2/7893_rail79k_R2ADi.dat}
\def \rradichip{./plots2/3024_chip0_R2ADi.dat}
\def \rradilung{./plots2/7893_lung2_R2ADi.dat}
\centering
\begin{tikzpicture}% residual norms
	\begin{semilogyaxis}[
		%title={foo}, % we use the caption of the surrounding figure as substitute for title
		ymax = 9,
		%ymin = 1e-10,
		xlabel = {\scriptsize subspace dimension},
		ylabel = {\scriptsize $\lVert R_j^* R_j \rVert_{2} / \lVert CC^*\rVert_{2}$},
		legend entries = {rail79k, chip0, lung2},
		legend pos = north east,
		width=.8\textwidth,
		]
		\addplot table[x=subspace_dimension,y=relative_residual]{\rradirail};
		\addplot table[x=subspace_dimension,y=relative_residual]{\rradichip};
		\addplot table[x=subspace_dimension,y=relative_residual]{\rradilung};
	\end{semilogyaxis}
\end{tikzpicture}%
%
\caption{Relative residual norms for the rail79k, chip0 and lung2 example.}\label{fig:residual}
\end{figure}
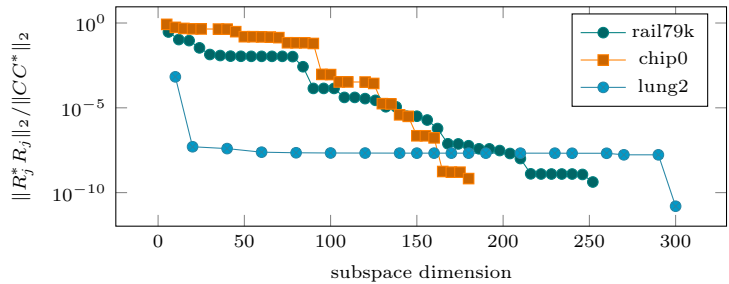

The convergence of the iteration considerably depends on the choice of shifts which determine the Krylov subspace used for the approximate solution. We used the M-M.E.S.S.-2.0 implementation of the RADI iteration with the shift strategy \emph{residual Hamiltonian shifts} as described in \Cref{sec:shifts} (denoted \emph{gen-ham-opti} in M-M.E.S.S.-2.0) and parameter $l=6p$ to precompute the shifts which were then used in our numerical experiments. The RADI iteration was stopped when the relative residual norm $\lVert R_j R_j^*\rVert_2/\lVert C^* C\rVert_2 = \lVert R_j^* R_j\rVert_2/\lVert C C^*\rVert_2$ became smaller than  $10^{-9}$. Due to the precomputed shifts the Riccati ADI solution is fixed. This allows us to employ our parallelization approach and compare the accuracy of approximate solutions obtained with different numbers of parallel threads. However, as shifts are precomputed, no timings for the shift calculation are presented, although this may contribute considerably to the iteration time. For a discussion of the effects of different shift strategies and different Riccati ADI methods we refer to the numerical experiments in \cite[Sec. 5]{benner18}.

We used the following three examples. The rail example (ID 1445, \cite{bennerSaak05}) describes the semi-discretization of a heat transfer process for optimal cooling of steel profiles. It consists of symmetric negative/positive definite $A,E$ and was used with $n=79841$, $m=7$ and $p=6$. It is denoted by \emph{rail79k}. The second example is the \emph{chip0} example (ID 1428, \cite{moosmann04}), a finite element model of a chip cooled by convection. Its system matrices have sizes $n=20082,$ $m=1$ and $p=5$. Both originate from the Oberwolfach Benchmark Collection \cite{oberwolfach05}. Further, we used the example \emph{lung2}, modeling processes in the human lung, with $n=109460,$ and $E=I$ from the UF Sparse Matrix Collection \cite{ufsparse11}. We employed this example with the negated system matrix $-A$, $m=p=10$ and $C=B^*$ chosen at random. All these examples are real valued, so all iteration steps were executed with realification in case of complex shifts.

\subsection{Comparison of R\textsuperscript{2}ADi and Lyapunov RADI iteration}
All numerical experiments in this subsection were executed using MATLAB 2019a on an Intel\textsuperscript{\textregistered} Core\texttrademark\, i7-5600U CPU @ 2.60 GHz with 12 GB RAM.

\begin{table}[b]
	\small
\begin{tabular}{lrrrrrrrr} 
\toprule
& \hspace{-1cm} subsp. &\multicolumn{4}{l}{R\textsuperscript{2}ADi}&\multicolumn{3}{l}{Lyap. RADI}\\  
\cmidrule(lr){3-6} \cmidrule(lr){7-9}
& dim. &  RAD exp.& $ZY_{12}$ mul. & misc. & total & RAD exp.& misc.& total\\ 
\midrule 
rail79k & 252 & 7.8 & 0.5 & 0.8 & \textbf{9.1} & 8.8 & 0.8 & \textbf{9.6} \\
chip0 & 180 & 32.8 & 0.1 & 0.1 & \textbf{33.0} & 34.9 & 0.1 & \textbf{35.0}\\
lung2 & 300 & 10.0 & 0.4 & 1.0 & \textbf{11.4} & 16.0 & 1.0 & \textbf{17.0}\\
\bottomrule
\end{tabular}
\caption{Times in seconds for different parts of the iterations.}\label{tab:timing}
\end{table}

\begin{figure}[t]
\centering
\def \rradirail{./plots2/7893_rail79k_R2ADi.dat}
\def \lradirail{./plots2/7893_rail79k_LRADI.dat}
\def \rradichip{./plots2/3024_chip0_R2ADi.dat}
\def \lradichip{./plots2/3024_chip0_LRADI.dat}
\def \rradilung{./plots2/7893_lung2_R2ADi.dat}
\def \lradilung{./plots2/7893_lung2_LRADI.dat}
\begin{tabular}{rl}
	\begin{tikzpicture}[baseline,trim axis left]% timings
	\begin{axis}[
		%title={foo}, % we use the caption of the surrounding figure as substitute for title
		ymax = .28,
		ymin = 0,%-.02,
		ylabel = {\scriptsize time in seconds},
		%legend entries = {RAD exp., $ZY_{12}$ mult.},
		%legend pos = north east,
		title = {R\textsuperscript{2}ADi},
		]
		\addplot table[x=iteration_no,y=expand_rad]{\rradirail};
			\label{plt:expansion}
		\addplot table[x=iteration_no,y=ZmultY12]{\rradirail};
			\label{plt:ZYmult}
		\node[anchor=north east] at (rel axis cs:1,1) {rail79k};
	\end{axis}
\end{tikzpicture}%
&
\begin{tikzpicture}[baseline,trim axis right]% timings
	\begin{axis}[
		%title={foo}, % we use the caption of the surrounding figure as substitute for title
		ymax = .28,
		ymin = 0,%-.02,
		%xlabel = {\scriptsize iteration number},
		%ylabel = {\scriptsize time in seconds},
		%legend entries = {RAD exp.},
		%legend pos = north east,
		yticklabel pos=right,
		title = {Lyapunov RADI},
		]
		\addplot table[x=iteration_no,y=expand_rad]{\lradirail};
		\node[anchor=north east] at (rel axis cs:1,1) {rail79k};
	\end{axis}
\end{tikzpicture}%
\\
\begin{tikzpicture}[baseline,trim axis left]% timings
	\begin{axis}[
		%title={foo}, % we use the caption of the surrounding figure as substitute for title
		ymax = 2.5,
		ymin = 0,%-.02,
		ylabel = {\scriptsize time in seconds},
		%legend entries = {RAD exp., $ZY_{12}$ mult.},
		%legend pos = north east,
		]
		\addplot table[x=iteration_no,y=expand_rad]{\rradichip};
		\addplot table[x=iteration_no,y=ZmultY12]{\rradichip};
		\node[anchor=north east] at (rel axis cs:1,1) {chip0};
	\end{axis}
\end{tikzpicture}%
&
\begin{tikzpicture}[baseline,trim axis right]% timings
	\begin{axis}[
		%title={foo}, % we use the caption of the surrounding figure as substitute for title
		ymax = 2.5,
		ymin = 0,%-.02,
		%xlabel = {\scriptsize iteration number},
		%ylabel = {\scriptsize time in seconds},
		%legend entries = {RAD exp.},
		%legend pos = north east,
		yticklabel pos=right,
		]
		\addplot table[x=iteration_no,y=expand_rad]{\lradichip};
		\node[anchor=north east] at (rel axis cs:1,1) {chip0};
	\end{axis}
\end{tikzpicture}%
\\
\begin{tikzpicture}[baseline,trim axis left]% timings
	\begin{axis}[
		%title={foo}, % we use the caption of the surrounding figure as substitute for title
		ymax = 1.5,
		ymin = 0,%-.02,
		xlabel = {\scriptsize iteration number},
		ylabel = {\scriptsize time in seconds},
		%legend entries = {RAD exp., $ZY_{12}$ mult.},
		%legend pos = north east,
		]
		\addplot table[x=iteration_no,y=expand_rad]{\rradilung};
		\addplot table[x=iteration_no,y=ZmultY12]{\rradilung};
		\node[anchor=north east] at (rel axis cs:1,1) {lung2};
	\end{axis}
\end{tikzpicture}%
&
\begin{tikzpicture}[baseline,trim axis right]% timings
	\begin{axis}[
		%title={foo}, % we use the caption of the surrounding figure as substitute for title
		ymax = 1.5,
		ymin = 0,%-.02,
		xlabel = {\scriptsize iteration number},
		%ylabel = {\scriptsize time in seconds},
		%legend entries = {RAD exp.},
		%legend pos = north east,
		yticklabel pos=right,
		]
		\addplot table[x=iteration_no,y=expand_rad]{\lradilung};
		\node[anchor=north east] at (rel axis cs:1,1) {lung2};
	\end{axis}
\end{tikzpicture}%
\end{tabular}%
\caption{Computational time in each iteration step of the R\textsuperscript{2}ADi and Lyapunov RADI iteration for the RAD expansion (\ref*{plt:expansion}) and the $ZY_{12}$ multiplication (\ref*{plt:ZYmult}).}\label{fig:timing}
\end{figure}
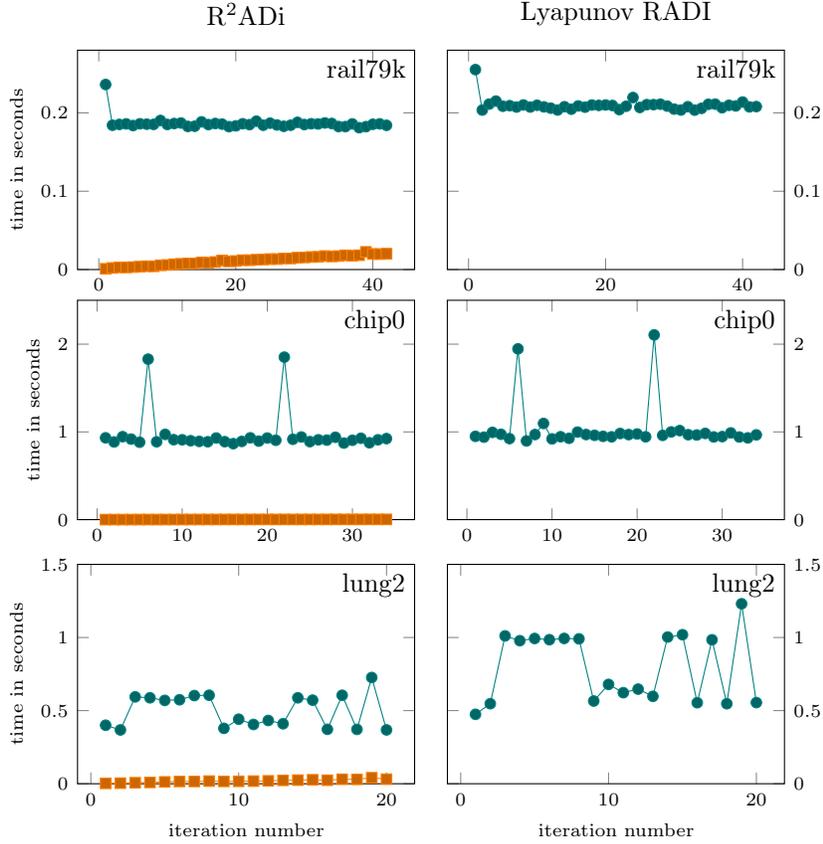

In \Cref{fig:residual} the convergence behaviour for the three examples is displayed. In \Cref{tab:timing} we display the computational time for the different parts of the two algorithms compared. The expansion of the RADs is the most expensive part due to the necessary solves of linear systems. In the R\textsuperscript{2}ADi we additionally have to multiply the current Krylov basis with the variable $Y_{12}$ in line \ref{lin:ZY12} of \Cref{alg:R2ADi}. All other parts of the iterations are aggregated under \emph{misc}.

\Cref{fig:timing} shows the cost for the RAD expansion and the $ZY_{12}$ multiplication in every step of the iterations. While the cost for solving the linear systems is almost constant during the iteration (the spikes are due to the more expensive solves with complex shifts), the $ZY_{12}$ multiplication becomes more expensive from step to step as the number of columns in the Krylov basis $Z$ grows by $p$ columns in every step. We can also observe that the Lyapunov RADI iteration needs more time for the linear solves, as due to the SMW formula equations with $m+p$ right hand sides have to be solved instead of only $p$ right hand sides in the R\textsuperscript{2}ADi. For instance in the lung2 example with $m=10$ the additional costs for the SMW formula make system solves $60\%$ more expensive than in the R\textsuperscript{2}ADi, while in the chip0 example with $m=1$ the additional costs are only little.

Due to the linearly increasing cost of the $ZY_{12}$ multiplication it is advantageous to switch from the R\textsuperscript{2}ADi to the Lyapunov RADI iteration as soon as the steps of the Lyapunov RADI iteration become cheaper than the steps of R\textsuperscript{2}ADi. However, in all examples considered here the cheaper system solves in the R\textsuperscript{2}ADi compensate for the additional time needed for the $ZY_{12}$ multiplication.

\subsection{Effect of parallelization in R\textsuperscript{2}ADi}
For the experiments in this subsection we used MATLAB 2018b on four Intel\textsuperscript{\textregistered} Xeon\textsuperscript{\textregistered}\, CPU E7-4880 v2 @ 2.50 GHz with altogether 60 CPU cores and 1~TB RAM. For the parallel expansion of the RADs the \emph{parfor} command in MATLAB was utilized. All calculations were performed with the Riccati RAD iteration \Cref{alg:R2ADi}.

\begin{figure}
\centering
\def \rradirail{./plots2/6519_rail79k_R2ADi_parallel.dat}
\def \rradichip{./plots2/6519_chip0_R2ADi_parallel.dat}
\def \rradilung{./plots2/6519_lung2_R2ADi_parallel.dat}
\begin{tikzpicture}% timings
	\begin{axis}[
		%title={foo}, % we use the caption of the surrounding figure as substitute for title
		ymax = 30,
		ymin = 0,%-.02,
		xlabel = {\scriptsize parallel threads},
		ylabel = {\scriptsize time in s},
		legend entries = {rail79k, chip0, lung2},
		legend pos = north east,
		]
		\addplot table[x=parallel_no,y=all]{\rradirail};
		\addplot table[x=parallel_no,y=all]{\rradichip};
		\addplot table[x=parallel_no,y=all]{\rradilung};
	\end{axis}
\end{tikzpicture}%
\begin{tikzpicture}% timings
	\begin{axis}[
		%title={foo}, % we use the caption of the surrounding figure as substitute for title
		ymax = 10,
		ymin = 0,%-.02,
		xlabel = {\scriptsize parallel threads},
		ylabel = {\scriptsize speedup},
		legend entries = {rail79k, chip0, lung2},
		legend pos = north west,
		]
		\addplot table[x=parallel_no,y=speedup]{\rradirail};
		\addplot table[x=parallel_no,y=speedup]{\rradichip};
		\addplot table[x=parallel_no,y=speedup]{\rradilung};
	\end{axis}
\end{tikzpicture}%
%
\caption{Times and speedups for R\textsuperscript{2}ADi with parallel RAD expansion.}\label{fig:parallel_timing}
\end{figure}
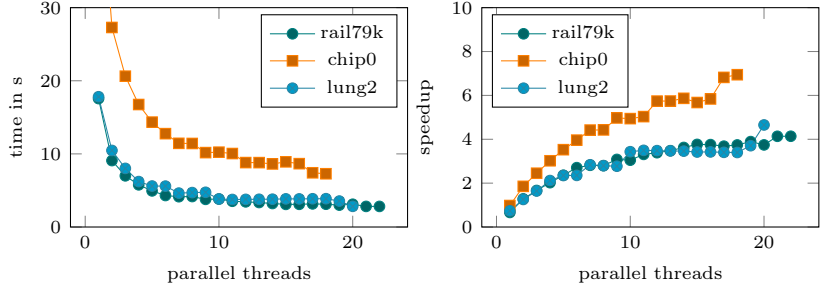
To parallelize the system solves, multiple shifts have to be available. This is the case here as all shifts were precomputed, which allows us to compare the performance and accuracy of the parallel with the serial approach. In practice the shifts are calculated one after another during the iteration as described in \Cref{sec:shifts}. An efficient shift strategy which obtains multiple shifts per iteration step has yet to be found.

In \Cref{fig:parallel_timing} the necessary times for the iteration and speedup factors are plotted against the number of parallel system solves in the RAD expansion step. The speedup factor is the iteration time for an iteration without parallelization and a \emph{for} loop divided by the iteration time needed with parallel system solves and MATLABs \emph{parfor} loop. Due to the overhead introduced by the \emph{parfor} command the serial iteration with a \emph{for} loop is faster than the same serial iteration with the \emph{parfor} loop. Thus the speedup factor for one thread is smaller than one. Further, only the system solves are executed in parallel but not the $ZY_{12}$ multiplication and miscellaneous tasks. As they take up to about $15\%$ of the calculation time in the rail79k and lung2 examples the possible maximal speedup is quite limited. Indeed we observe a moderate speedup for the rail79k and lung2 examples. For the chip0 example the speedup is higher. The factor for four parallel threads is $3.0$ and it increases to $4.4$ when eight parallel threads are used.

Besides the performance gain we also investigate the accuracy of the parallel calculations. We therefore compare the parallel iterates with the serial iterates. Let a subscript $(k)$ denote the number of parallel threads utilized to obtain the variable. We calculated the relative deviation of the residuals obtained with $k$ parallel threads from the residual obtained with the serial iteration, i.e.
\begin{align}
	\frac{\lVert \mathcal{R}_{(1)} - \mathcal{R}_{(k)} \rVert_2}{ \lVert \mathcal{R}_{(1)} \rVert_2}.
\end{align}
The second quantity we use to indicate the accuracy is the relative deviation of the parallel approximants $X_{(k)}$ from the serial approximant $X_{(1)}$
\begin{align}
	\frac{\lVert X_{(1)}-X_{(k)} \rVert_{2}}{\lVert X_{(1)} \rVert_{2}}.
\end{align}
A direct calculation of the norms of the involved matrices is infeasible due to the large dimensions. We thus exploit the factorized form of the residual and the approximants, i.e. $\mathcal{R}_{(k)}=R_{(k)}R_{(k)}^{*}$ and $X_{(k)}=Z_{(k)}Z_{(k)}^{*}$. Let $QS=\begin{bmatrix}Z_{(1)}& Z_{(k)}\end{bmatrix}$ be an economy-size QR decomposition with upper triangular $S$. Then due to the unitary invariance of the norm it holds
\begin{align}
	\lVert X_{(1)}-X_{(k)} \rVert_2 = \left\lVert S \begin{bmatrix}I&0\\0&-I\end{bmatrix}S^* \right\rVert_2
\end{align}
and only the norm of a small matrix must be computed. For the denominator we use the formula $\lVert Z_{(k)}Z_{(k)}^{*} \rVert_2=\lVert Z_{(k)}^{*}Z_{(k)}\rVert_2$. We proceed in the same way for the residual deviation.

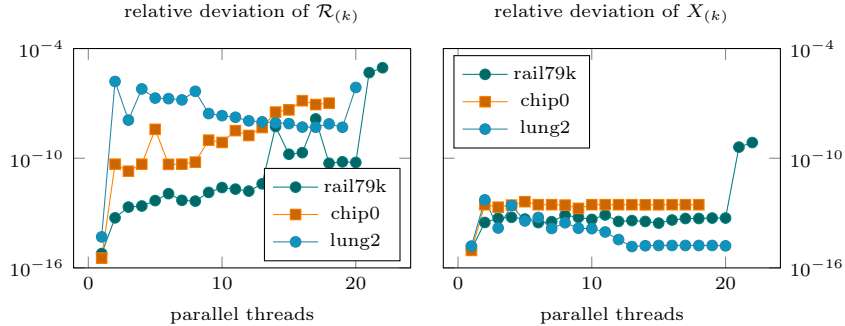
\begin{figure}
\centering
\def \rradirail{./plots2/6519_rail79k_R2ADi_parallel.dat}
\def \rradichip{./plots2/6519_chip0_R2ADi_parallel.dat}
\def \rradilung{./plots2/6519_lung2_R2ADi_parallel.dat}
\begin{tabular}{lr}
\begin{tikzpicture}% accuracy
	\begin{semilogyaxis}[
		title={\scriptsize relative deviation of $\mathcal{R}_{(k)}$},
		ymax = 1e-4,
		ymin = 1e-16,
		xlabel = {\scriptsize parallel threads},
		legend entries = {rail79k, chip0, lung2},
		legend pos = south east,
		]
		\addplot table[x=parallel_no,y=relnormR1minusRjqr]{\rradirail};
		\addplot table[x=parallel_no,y=relnormR1minusRjqr]{\rradichip};
		\addplot table[x=parallel_no,y=relnormR1minusRjqr]{\rradilung};
	\end{semilogyaxis}
\end{tikzpicture}%
&
\begin{tikzpicture}% accuracy 
	\begin{semilogyaxis}[
		title={\scriptsize relative deviation of $X_{(k)}$},
		ymax = 1e-4,
		ymin = 1e-16,
		xlabel = {\scriptsize parallel threads},
		legend entries = {rail79k, chip0, lung2},
		legend pos = north west,
		yticklabel pos=right,
		]
		\addplot table[x=parallel_no,y=relnormX1minusXjqr]{\rradirail};
		\addplot table[x=parallel_no,y=relnormX1minusXjqr]{\rradichip};
		\addplot table[x=parallel_no,y=relnormX1minusXjqr]{\rradilung};
	\end{semilogyaxis}
\end{tikzpicture}%
\end{tabular}%
\caption{Relative deviation of parallel residuals and approximants from serial results.}\label{fig:parallel_accuracy}
\end{figure}

The results are displayed in \Cref{fig:parallel_accuracy}. We observe that the relative deviations from the residual are below $10^{-5}$ for all examples. The relative deviation of the final approximation is even smaller than $10^{-12}$ for all examples and fewer than twenty parallel threads. All in all the effects of parallelization to the accuracy of the results appear to be very little. Thus parallel system solves are a reasonable technique to speed up computations.

\section{Conclusion}\label{sec:conclusion}
In this work we have introduced a new approach for the ADI-type approximation of complex large-scale algebraic Riccati equations. We made use of rational Krylov decompositions to rewrite the Riccati residual. By imposing a rank condition on this formulation of the residual we obtained a small scale Lyapunov equation, which characterizes the sought solution. It was shown that the Riccati ADI approximate solution exists and is unique under a simple condition for the shifts. Uniqueness of the solution implies the equivalence of all previously known Riccati ADI methods and our new approach. Further, we revealed that the Riccati ADI solution can be interpreted as an oblique projection onto a rational Krylov subspace if the kernel of the projection contains the Riccati residual factor. The residual factor is a rational function. The poles and zeros of this function were connected to the eigenvalues of projected system matrices.

We introduced two new iterative methods to calculate the Riccati ADI solution. Both make use of the fact that the solution of the small scale Lyapunov equation can be updated efficiently. In our first iterative method, the Riccati RAD iteration, only system solves with matrices of the form $A^*-\mu I$ and the residual factor are necessary. The second method derived, the Lyapunov RADI iteration, contains the RADI iteration as a special case. Here, system solves with a matrix of the form $A^*-KB^* -\mu I$ and the residual factor are necessary. In both algorithms the extension of the Krylov subspace was decoupled from the rest of the iteration, which made parallelization and, in case of real system matrices, realification possible easily.

The numerical experiments show the competitiveness of our new approach. They indicate that it is beneficial to start with the R\textsuperscript{2}ADi and switch to the Lyapunov RADI iteration for best performance whenever the R\textsuperscript{2}ADi becomes too expensive due to the linear increasing cost of the $ZY_{12}$ multiplication. Parallel system solves scale well with the number of parallel threads if they dominate the iteration, even though there seems to be a large overhead due to the use of MATLABs \emph{parfor} command. The accuracy of the solution obtained with parallelization is remarkably good in all numerical examples, even for as many as 20 parallel solves. However, to make the parallel approach work in practice, a shift strategy has to be found which generates multiple shifts during the iteration.

\begin{acknowledgements}
We thank Jens Saak, Patrick Kürschner and Zvonimir Bujanović for the discussions about shift strategies and for sharing their RADI implementation with us.
\end{acknowledgements}

\small
\bibliography{literatur}{}
\bibliographystyle{plain}
\end{document}